\documentclass[12pt]{amsart}
\usepackage[margin=1.2in]{geometry}
\usepackage{amssymb,amsfonts,latexsym,amscd}
\usepackage[dvipsnames]{xcolor}
\usepackage[all]{xy}
\usepackage{verbatim}
\usepackage{hyperref}
\usepackage{mathrsfs}
\usepackage{mathtools}
\usepackage{amssymb}
\usepackage{stmaryrd}
\usepackage{tikz-cd}

\hypersetup{
    colorlinks=true, 
    linkcolor=Blue, 
    urlcolor=Blue, 
    citecolor=Blue, 
    filecolor=Blue, 
    pdfborder={0 0 0} 
}

\newtheorem{theorem}{Theorem}[section]
\newtheorem{lemma}[theorem]{Lemma}
\newtheorem{proposition}[theorem]{Proposition}
\newtheorem{corollary}[theorem]{Corollary}
\theoremstyle{definition}

\newtheorem{example}[theorem]{Example}

\newtheorem{remark}[theorem]{Remark}

\numberwithin{equation}{section}

\newcommand{\ot}{\otimes}

\newcommand{\B}{\mathcal{B}}

\newcommand{\cS}{\mathcal{S}}

\newcommand{\Hom}{\operatorname{Hom}}

\newcommand{\id}{\operatorname{id}}

\newcommand{\gen}[1]{{\left\langle #1 \right\rangle}}

\newcommand{\im}{\text{im}}

\renewcommand{\bar}[1]{\overline{#1}}

\newcommand{\Rep}{\operatorname{Rep}}

\renewcommand{\Vec}{\operatorname{Vec}}

\newcommand{\coadj}{\rightharpoonup\mathrel{\mspace{-15mu}}\rightharpoonup}
\newcommand{\ad}{\operatorname{ad}}

\begin{document}
	
\title[On the Drinfeld double of a finite group scheme]{On the Drinfeld double of a finite group scheme and its representation category}

\author{Daniel Arreola}
\address{Department of Mathematics,
Iowa State University, Ames, IA 50100, USA.}
\email{darre@iastate.edu}

\author{Shlomo Gelaki}
\address{Department of Mathematics,
Iowa State University, Ames, IA 50100, USA.}
\email{gelaki@iastate.edu}

\date{\today}


\begin{abstract}
We classify equivalence classes of Hopf algebra quotient pairs $(D,\theta)$ of the Drinfeld
double $D(G)$ of a finite group scheme $G$ over an algebraically closed field
$\mathbf{k}$ of characteristic $p\ge 0$, in terms of group scheme-theoretical data. We prove that such Hopf algebra quotients $D$ are Hopf algebra extensions $\mathscr{O}(K)^{\mathrm{cop}}\#_{\sigma}^{\tau} \mathbf{k}[G/H]$, where $K$ and $H$ are normal subgroup schemes of $G$ that centralize each other and $B:\mathbf{k}[H]\to \mathscr{O}(K)$ is a $G$-equivariant Hopf algebra map, and describe the surjective Hopf algebra map $\theta:D(G)\twoheadrightarrow D$. Using this classification, we determine the tensor subcategories of the center $\mathscr{Z}(G):=\Rep(D(G))$ of $G$, describe their centralizers, determine when they are symmetric or non-degenerate, and give a description of their simple and projective objects using \cite{GS}. 
Our categorical results generalize those found in \cite{NNW} in characteristic $0$.
\end{abstract}

\maketitle

\tableofcontents

\section{Introduction}\label{sec:Introduction}
The representation category $\mathscr{Z}(G):=\Rep(D(G))$ of the Drinfeld double $D(G)$ of a finite group scheme $G$ over an algebraically closed field $\mathbf{k}$ of characteristic $p\ge 0$ plays a central role in the theory of finite braided tensor categories 
\cite{CH,CCC,DGNO,ERW,FN1,FN2,Ge,Ge2,GS,GS2,Go,N,NNW}. In characteristic $p=0$, the category $\mathscr{Z}(G)$ is a non-degenerate braided {\em fusion} category; in this semisimple setting the tensor subcategories of $\mathscr{Z}(G)$ were classified by Naidu, Nikshych, and Witherspoon \cite{NNW}. A key feature of \cite{NNW} is that \emph{modular data} can be used as an effective organizing principle. Namely, the $S$-matrix of $\mathscr{Z}(G)$ detects when two simple objects of $\mathscr{Z}(G)$ centralize each other, and M\"uger centralizers can be computed from $S$-matrix relations, ultimately translating the subcategory problem into explicit group-theoretic conditions involving commuting normal subgroups of $G$.

In positive characteristic $p>0$, $\mathscr{Z}(G)$ is typically \emph{not semisimple}. While one still has a robust notion of non-degeneracy for finite braided tensor categories (e.g. in the sense of Lyubashenko, equivalent to factorizability in the sense of Etingof--Nikshych--Ostrik by Shimizu \cite{Sh}), there is in general no computable $S$-matrix attached to the simple objects that could play the same role as in the fusion case. Thus, the strategy of \cite{NNW} does not directly extend to the non-semisimple setting.

Our approach replaces $S$-matrix technology with a \emph{Hopf-theoretic} classification of Hopf quotient pairs of $D(G)$, and then uses the fact that Hopf quotient pairs of $D(G)$ encode tensor subcategories of $\mathscr{Z}(G)$. Indeed, tensor subcategories of $\Rep(H)$, for a finite dimensional Hopf algebra $H$, correspond to equivalence classes of Hopf algebra quotient pairs of $H$ (see, e.g. \cite[Proposition 2]{BN}), so one can study tensor subcategories of $\mathscr{Z}(G)=\Rep(D(G))$ by classifying Hopf algebra quotient pairs of $D(G)$. This perspective is particularly well-suited to positive characteristic: although $S$-matrices are unavailable, the braiding on $\Rep(D(G))$ is still controlled by the universal $R$-matrix of $D(G)$, and centralizing conditions can be tested on the Hopf algebra side via the element $R_{21}R$ (cf. the criterion used in the proof of Theorem \ref{main2} below).

The first main result of this paper is a classification of Hopf algebra quotient pairs $(D,\theta)$ of $D(G)$ in terms of group scheme-theoretical data. We show that every Hopf algebra quotient $D$ of $D(G)$ is of the form $$D(K,H,B):=\mathscr{O}(K)^{\mathrm{cop}}\#^{\tau}_{\sigma} \mathbf{k}[G/H],$$ where $K$ and $H$ are normal subgroup schemes of $G$ that centralize each other and $B\colon \mathbf{k}[H]\to \mathscr{O}(K)$ is a $G$-equivariant Hopf algebra map, and describe the surjective Hopf algebra map $\theta:D(G)\twoheadrightarrow D$ (Theorems \ref{quotientha} and \ref{uniqueness}). Moreover, $D(K,H,B)$ is always ribbon braided, and we give explicit formulas for its $R$-matrix and ribbon element (Corollary \ref{main0}). We also determine precise criteria for $D(K,H,B)$ to be triangular or factorizable, yielding large families of quasitriangular Hopf algebras realized as Hopf algebra quotients of $D(G)$.

The second main result is a complete, group scheme-theoretical description of the tensor subcategory lattice of $\mathscr{Z}(G)$ and its centralizer theory. Using the above Hopf algebra quotient pairs classification, we show that the assignment 
$$(K,H,B)\mapsto \mathscr{Z}(K,H,B):=\Rep(D(K,H,B))$$
yields a bijection between triples $(K,H,B)$ and tensor subcategories of $\mathscr{Z}(G)$. We then compute M\"uger centralizers inside $\mathscr{Z}(G)$: for each tensor subcategory $\mathscr{Z}(K,H,B)$, its centralizer is $\mathscr{Z}(H,K,\overline{B})$ (Theorem \ref{main2}), generalizing the characteristic-$0$ picture of \cite{NNW} from modular-data methods to a purely Hopf-theoretic criterion. As consequences, we obtain explicit criteria for symmetry, non-degeneracy, and Lagrangianity of tensor subcategories of $\mathscr{Z}(G)$ (Corollary \ref{main4}). 

Finally, using the geometric description of $\mathscr{Z}(G)$ from \cite{GS}, we describe the simple and projective objects in $\mathscr{Z}(K,H,B)$ and compute their Frobenius-Perron dimensions (Theorem \ref{main1}). In \S\ref{sec:examples}, we discuss some special cases and examples illustrating our results. In particular, our factorizability criteria produce families of non-degenerate finite braided tensor categories in positive characteristic (coming from factorizable Hopf algebra quotients of $D(G)$), and we exhibit examples in both the constant and connected cases.

\begin{remark}\label{future}
The results of this paper can be extended to the twisted Drinfeld double $D^{\omega}(G)$ of a finite group scheme $G$ (it is a quasi-Hopf algebra) and its representation category $\Rep(D^{\omega}(G))$, by using similar ideas and results from \cite{Ge, GS2}. In characteristic $p=0$, this was done in \cite[Section 5]{NNW}. \qed
\end{remark}

\subsection{Organization}
In \S2 we recall background on finite group schemes, Hopf algebras, and Drinfeld doubles. In \S3 we construct and classify Hopf algebra quotient pairs of $D(G)$ and analyze the induced (quasitriangular ribbon) structures on $D(K,H,B)$. In \S4 we apply these results to classify tensor subcategories of $\mathscr{Z}(G)$, compute their centralizers, and derive criteria for symmetry, nondegeneracy, and Lagrangianity. We then describe the simples and projectives in $\mathscr{Z}(K,H,B)$ using \cite{GS}. 

\subsection{Acknowledgements} 
The work of S.G. was supported by Simons Foundation Award 963288.

\section{Preliminaries}\label{sec:Preliminaries}

We work over an algebraically closed field $\mathbf{k}$ of characteristic $p\ge0$. 
We assume familiarity with the theory of finite tensor categories and finite group schemes over $\mathbf{k}$, and refer to \cite{EGNO,J,W} for any unexplained notion. 

\subsection{Finite group schemes}\label{sec:Finite group schemes}
A finite group scheme $G$ over $\mathbf{k}$ is a finite scheme over $\mathbf{k}$ whose coordinate algebra $\mathscr{O}(G)$ is a finite dimensional commutative Hopf algebra (see, e.g., \cite{J,W}), so that its {\em group algebra} $\mathbf{k}[G]:=\mathscr{O}(G)^*$ is a finite dimensional cocommutative Hopf algebra. We set $|G|:=\dim_{\mathbf{k}}(\mathscr{O}(G))=\dim_{\mathbf{k}}(\mathbf{k}[G])$. 

Let $G^{\circ}$ be the identity component of $G$, and let $G(\mathbf{k})\subseteq G$ be the subgroup of closed points of $G$. Recall that we have a split exact sequence of finite group schemes
\begin{equation}\label{sesG}
1\to G^{\circ}\xrightarrow{} G \mathrel{\mathop{\rightleftarrows}^{\pi}_{\gamma}} G(\mathbf{k})\to 1.
\end{equation}

Let $L$ be a closed subgroup scheme of $G$, let 
\begin{equation}\label{iotaL}
\iota=\iota_L=\iota_{L,G}:L\hookrightarrow G
\end{equation}
be the inclusion of group schemes, and let 
\begin{equation}\label{qL}
q=q_L=q_{G,L}=\iota_{L}^{\sharp}:\mathscr{O}(G)\twoheadrightarrow\mathscr{O}(L)
\end{equation}
be the corresponding surjective Hopf algebra map. Recall \cite[Theorem 2.2]{S} that we can choose a section 
\begin{equation}\label{muL}
\mu=\mu_L:\mathscr{O}(L)\xrightarrow{1:1} \mathscr{O}(G)
\end{equation}
for $q$, so that $\mu$ is an $\mathscr{O}(L)$-colinear map; that is, $\varepsilon \mu = \varepsilon$ and 
\begin{equation}\label{mucolinear}
\mu(a_1)\ot a_2=\mu(a)_1\ot q(\mu(a)_2);\quad\forall a\in \mathscr{O}(L),
\end{equation}
such that $\mu(1)=1$ and 
\begin{equation}\label{qLmuLsection}
q(\mu(a))=a;\quad \forall a\in \mathscr{O}(L).
\end{equation}
We will denote the dimension of the algebra $\mathscr{O}(G/L)=\mathscr{O}(G)^L$ by $[G:L]$.

Recall that the {\em right adjoint} (or conjugation) action of $G$ on itself is given by
\begin{equation}\label{right adjoint action}
{\rm ad_r}:G\times G\to G,\quad (g,f)\mapsto f^{-1}gf;
\end{equation}
that is, ${\rm ad_r}(u)(\tilde{u})=S(u_1)\tilde{u}u_2$ for $u,\tilde{u}\in\mathbf{k}[G]$, 
and the algebra comorphism
\begin{equation}\label{coadog}
{\rm ad_r}^{\sharp}:\mathscr{O}(G)\to\mathscr{O}(G)\ot\mathscr{O}(G),\quad b\mapsto b_2\ot S(b_1)b_3,
\end{equation}
is the {\em right adjoint} coaction of $\mathscr{O}(G)$ on itself. 
Then ${\rm ad_r}^{\sharp}$ corresponds to the \textit{left coadjoint} action $\coadj$ of $\mathbf{k}[G]$ on $\mathscr{O}(G)$, given by
\begin{equation}\label{left coadjoint action}
\mathbf{k}[G]\otimes\mathscr{O}(G)\to \mathscr{O}(G),\quad u\otimes b\mapsto u \coadj b := \left\langle S(b_1)b_3, u \right\rangle b_2.
\end{equation} 
Namely, for every $u,\tilde{u}\in \mathbf{k}[G]$ and $b\in \mathscr{O}(G)$, we have 
\begin{equation}\label{left coadjoint action2}
\langle u \coadj b,\tilde{u}\rangle=\langle b,{\rm ad_r}(u)(\tilde{u})\rangle.
\end{equation}
Clearly, if $G$ is commutative (that is, $\mathscr{O}(G)$ is cocommutative), then $\coadj$ is trivial.

\subsection{Normal subgroup schemes}\label{sec:Normal subgroup schemes}
Fix a finite group scheme $G$ over $\mathbf{k}$ as in \S\ref{sec:Finite group schemes}. Let $H\subseteq G$ be a {\em normal} subgroup scheme; that is, ${\rm ad_r}(H\times G)\subseteq H$, so that $G/H$ is a finite group scheme. Let 
$\pi=\pi_H:G\twoheadrightarrow G/H$ be the quotient group scheme morphism. By \cite[Theorem 2.2]{S}, the exact sequence of Hopf algebras
\begin{equation}\label{gammaH}
\mathbf{k}\to \mathbf{k}[H]\xrightarrow{\iota} \mathbf{k}[G]\xrightarrow{\pi} \mathbf{k}[G/H] \to \mathbf{k}
\end{equation}
is {\em cleft}. That is, viewing $\mathbf{k}[G]$ as a {\em right} $\mathbf{k}[G/H]$-comodule via
$$\rho=\rho_H:\mathbf{k}[G]\to \mathbf{k}[G]\ot \mathbf{k}[G/H],\quad u\mapsto u_1\ot\pi(u_2),$$
we have $\mathbf{k}[G]^{\text{co} \mathbf{k}[G/H]} = \mathbf{k}[H]$, and we can choose a section 
\begin{equation}\label{gamma}
\gamma=\gamma_H:\mathbf{k}[G/H]\xrightarrow{1:1} \mathbf{k}[G],
\end{equation}
which is a convolution invertible $\mathbf{k}[G/H]$-colinear map; that is, $\varepsilon \gamma = \varepsilon$ and 
\begin{equation}\label{gammacolinear}
\gamma(x_1)\ot x_2=\gamma(x)_1\ot \pi(\gamma(x)_2);\quad\forall x\in \mathbf{k}[G/H],
\end{equation}
such that 
$$\pi(\gamma(x))=x;\quad\forall x\in \mathbf{k}[G/H].$$

Consider the map $\eta=\eta_H:=\id\star(\gamma^{-1}\pi)$; that is,
\begin{equation}\label{etafromgamma}
\eta: \mathbf{k}[G] \to \mathbf{k}[H],\quad u\mapsto u_1\gamma^{-1}(\pi(u_2)).
\end{equation}
Note that since $\gamma$ and $\gamma^{-1}$ preserve the counit,  
$$\pi(u_1 \gamma^{-1}(\pi(u_2))) = \pi(u_1) S(\pi(u_2)) = \varepsilon(u) = \varepsilon(\gamma^{-1}\pi(u)) = \varepsilon(u_1 \gamma^{-1}(\pi(u_2))),$$
so indeed $\eta(u) \in \mathbf{k}[G]^{\text{co} \mathbf{k}[G/H]} = \mathbf{k}[H]$.

\begin{lemma}\label{gcphgh}
The following hold:
\begin{enumerate}
\item
$\eta$ is a $\mathbf{k}[H]$-linear retraction for the inclusion map $\iota:\mathbf{k}[H]\hookrightarrow \mathbf{k}[G]$, such that 
for every $x\in \mathbf{k}[G/H]$, we have $\eta(\gamma(x))=\varepsilon(x)$.
\item
For every $u\in \mathbf{k}[G]$, we have $u=\eta(u_1)\gamma(\pi(u_2))$.
\item
For every $u,\tilde{u}\in \mathbf{k}[G]$, we have  
$$\eta(u\tilde{u})=\eta(u_1)\gamma(\pi(u_2))_1\eta(\tilde{u}_1)S(\gamma(\pi(u_2))_2)\eta\{\gamma(\pi(u_2))_3\gamma(\pi(\tilde{u}_2))\}.$$
\item
For every $u\in \mathbf{k}[G]$, we have $\gamma(\pi(u))=\eta^{-1}(u_1)u_2$.
\item
For every $u\in \mathbf{k}[G]$, we have $\eta^{-1}(u)=\gamma(\pi(u_1))S(u_2)$.
\end{enumerate} 
\end{lemma}

\begin{proof}
(1) Using (\ref{gammacolinear}) we verify $\eta(\gamma(x)) = \gamma(x)_1 \gamma^{-1}(\pi(\gamma(x)_2)) = \gamma(x_1) \gamma^{-1}(x_2) = \varepsilon(x)$.

(2) Since $\eta= \id \star (\gamma^{-1}\pi)$, we have $\eta\star (\gamma \pi) = \id$.

(3) By (2), we have
\begin{eqnarray*}
\lefteqn{u\tilde{u}=\left\{\eta(u_1)\gamma(\pi(u_2))\right\}
\left\{\eta(\tilde{u}_1)\gamma(\pi(\tilde{u}_2))\right\}}\\
& = & \eta(u_1)\left\{\gamma(\pi(u_2))_1\eta(\tilde{u}_1)S(\gamma(\pi(u_2))_2)\right\}
\gamma(\pi(u_2))_3\gamma(\pi(\tilde{u}_2)),
\end{eqnarray*}
which implies that
$$\eta(u\tilde{u})=\eta(u_1)\left\{\gamma(\pi(u_2))_1\eta(\tilde{u}_1)S(\gamma(\pi(u_2))_2)\right\}\eta\left\{\gamma(\pi(u_2))_3\gamma(\pi(\tilde{u}_2))\right\},$$
as claimed.

(4)-(5) Straightforward.
\end{proof}

Recall that the Hopf algebra $\mathbf{k}[G/H]$ measures $\mathbf{k}[H]$ via $\cdot$, given by
$$\mathbf{k}[G/H]\ot \mathbf{k}[H]\to \mathbf{k}[H],\quad x\cdot u=\gamma(x_1)u\gamma^{-1}(x_2).$$
Recall also that the $\mathbf{k}$-linear map  
\begin{equation}\label{defnolsigma}
\overline{\sigma}: \mathbf{k}[G/H] \otimes \mathbf{k}[G/H] \to \mathbf{k}[H],\quad 
\overline{\sigma}(x,y)=\gamma(x_1)\gamma(y_1)\gamma^{-1}(x_2y_2),
\end{equation}
is a $2$-cocycle, and the $\mathbf{k}$-linear map
\begin{equation}\label{defnoltau}
\begin{split}
& \overline{\tau}: \mathbf{k}[G/H] \to \mathbf{k}[H]\otimes \mathbf{k}[H],\\ 
& \overline{\tau}(x)=
\eta^{-1}(\gamma(x)_1)_1\eta(\gamma(x)_2)\otimes \eta^{-1}(\gamma(x)_1)_2\eta(\gamma(x)_3),
\end{split}
\end{equation}
is a co-cocycle. For example,
to see that $\overline{\sigma}$ is well defined, note first that
$$\pi(\gamma(x_1)\gamma(y_1)\gamma^{-1}(x_2y_2)) = x_1y_1S(x_2y_2) = \varepsilon(xy).$$ On the other hand, using (\ref{gammacolinear}) we have 
\begin{eqnarray*}
\lefteqn{\varepsilon(\gamma(x_1)\gamma(y_1)\gamma^{-1}(x_2y_2))=
\varepsilon(\gamma(x)_1 \gamma(y)_1 \gamma^{-1}(\pi(\gamma(x)_2 \gamma(y)_2)))}\\ 
& = & \varepsilon(\gamma^{-1}(\pi(\varepsilon(\gamma(x)_1) \gamma(x)_2 \varepsilon(\gamma(y)_1)\gamma(y)_2))) \\
& = & \varepsilon(\gamma^{-1}(\pi(\gamma(x)\gamma(y))))
= \varepsilon(\gamma^{-1}(xy))= \varepsilon(xy).
\end{eqnarray*}
This shows that $\gamma(x_1)\gamma(y_1)\gamma^{-1}(x_2y_2) \in \mathbf{k}[G]^{\text{co} \mathbf{k}[G/H]} = \mathbf{k}[H]$, as required.

Finally, it is well known (see e.g. \cite[Section 3.2]{AD}) that the map 
$${\rm f}:\mathbf{k}[G]\to \mathbf{k}[H]\#_{\overline{\sigma}}^{\overline{\tau}} \mathbf{k}[G/H],\quad u\mapsto \eta\left(u_1\right)\# \pi\left(u_2\right),$$
is a Hopf algebra isomorphism, whose inverse is given by
$${\rm f}^{-1}:\mathbf{k}[H]\#_{\overline{\sigma}}^{\overline{\tau}}\mathbf{k}[G/H]\to \mathbf{k}[G],\quad v\# x\mapsto v\gamma(x).$$
For example, using (\ref{gammacolinear}) and the fact that $\eta$ is $\mathbf{k}[H]$-linear, we compute
\begin{eqnarray*}
\lefteqn{{\rm f}({\rm f}^{-1}(v\ot x))={\rm f}(v\gamma(x))=\eta\left(v_1\gamma(x)_1\right)\ot \pi\left(v_2\gamma(x)_2\right)}\\
& = & v_1\eta\left(\gamma(x)_1\right)\ot \varepsilon(v_2)\pi\left(\gamma(x)_2\right)= v\eta\left(\gamma(x)_1\right)\ot \pi\left(\gamma(x)_2\right)\\
& = & v\eta\left(\gamma(x_1)\right)\ot x_2=v\varepsilon(x_1)\ot x_2=v\ot x,
\end{eqnarray*}
and
\begin{eqnarray*}
\lefteqn{{\rm f}^{-1}({\rm f}(u))={\rm f}^{-1}(\eta\left(u_1\right)\ot \pi\left(u_2\right))=\eta\left(u_1\right)\gamma\left(\pi\left(u_2\right) \right)}\\
& = & u_1\gamma^{-1}(\pi(u_2))\gamma\left(\pi\left(u_3\right) \right)=u_1\varepsilon(\pi(u_2))=u,
\end{eqnarray*}
so ${\rm f}$ and ${\rm f}^{-1}$ are indeed inverse of each other.

\begin{remark}\label{newtriv}
(1) If $\gamma$ (\ref{gamma}) is an algebra map then $\overline{\sigma}$ (\ref{defnolsigma}) is trivial.

(2) If $\eta$ (\ref{etafromgamma}) is a coalgebra map, equivalently if $\gamma$ (\ref{gamma}) is a coalgebra map (e.g. if $H=1$, or $G=G(\mathbf{k})$ is constant), then $\overline{\tau}$ (\ref{defnoltau}) is trivial. Indeed, since $\mathbf{k}[G]$ is cocommutative, $\eta^{-1}=(\gamma\pi)\star S$ is also a coalgebra map (see Lemma \ref{gcphgh}(4)), so   
\begin{eqnarray*}
\lefteqn{\overline{\tau}(x)=
\eta^{-1}(\gamma(x)_1)_1\eta(\gamma(x)_2)\otimes \eta^{-1}(\gamma(x)_1)_2\eta(\gamma(x)_3)}\\
& = & \eta^{-1}(\gamma(x_1))_1\eta(\gamma(x_2))\otimes \eta^{-1}(\gamma(x_1))_2\eta(\gamma(x_3))\\
& = & \eta^{-1}(\gamma(x_1))_1\varepsilon(x_2)\otimes \eta^{-1}(\gamma(x_1))_2\varepsilon(x_3)\\
& = & \eta^{-1}(\gamma(x))_1\otimes \eta^{-1}(\gamma(x))_2=
\eta^{-1}(\gamma(x_1))\otimes \eta^{-1}(\gamma(x_2))\\
& = & \varepsilon(x_1)\otimes \varepsilon(x_2)=\varepsilon(x)1\ot 1,
\end{eqnarray*}
as claimed. \qed
\end{remark}

\subsection{The Drinfeld double of a finite group scheme}\label{sec:The Drinfeld double of a finite group scheme}
Fix a finite group scheme $G$ over $\mathbf{k}$ as in \S\ref{sec:Finite group schemes}. 
Recall that the Drinfeld double 
$$D(G)=\mathscr{O}(G)^{{\rm cop}} \bowtie \mathbf{k}[G]$$
of $\mathbf{k}[G]$ is the tensor coalgebra $\mathscr{O}(G)^{{\rm cop}} \otimes \mathbf{k}[G]$, equipped with the product rule 
$$(b \bowtie u)(b' \bowtie u') = b(u_1 \coadj b') \bowtie u_2u',$$
where $\coadj$ is defined in (\ref{left coadjoint action}).  
Recall that $D(G)$ is an involutive Hopf algebra, such that the natural maps 
$$\mathbf{k}[G] \hookrightarrow D(G),\quad {\rm and}\quad \mathscr{O}(G)^{{\rm cop}} \hookrightarrow D(G)$$ 
are Hopf algebra inclusions. Since for every $u \in \mathbf{k}[G]$ and $b \in \mathscr{O}(G)$,  
\begin{eqnarray*}
\lefteqn{(1\bowtie u_1) (b\bowtie 1) (1\bowtie  S(u_2))=(1 \bowtie u_1)(b \bowtie S(u_2))}\\
& = & (u_1 \coadj b) \bowtie u_2S(u_3)= (u_1 \coadj b) \bowtie \varepsilon(u_2) = (u \coadj b) \bowtie 1,
\end{eqnarray*} 
it follows that $\mathscr{O}(G)$ is a normal Hopf subalgebra of $D(G)$, so we have an exact sequence of Hopf algebras 
$$\mathbf{k}\rightarrow \mathscr{O}(G) \rightarrow D(G) \rightarrow D(G)/(D(G)\mathscr{O}(G)^+) \rightarrow \mathbf{k}.$$ Since the map
\begin{equation}\label{qkg}
D(G)/(D(G)\mathscr{O}(G)^+) \to \mathbf{k}[G],\quad [b \bowtie u]\mapsto \varepsilon(b)u,
\end{equation}
is a Hopf algebra isomorphism, we have an exact sequence of Hopf algebras 
$$\mathbf{k}\rightarrow \mathscr{O}(G) \rightarrow D(G) \rightarrow \mathbf{k}[G] \rightarrow \mathbf{k}.$$

Finally, let $\B$ be a basis for $\mathbf{k}[G]$, and let $\{\delta_u\mid u\in\B\}$ be the dual basis for $\mathscr{O}(G)$. Recall that the element 
\begin{equation}\label{rmatrix}
R:=\sum_{u \in \B} (1\bowtie u) \otimes (\delta_u\bowtie 1)
\end{equation}
is an $R$-matrix for $D(G)$, and $(D(G),R)$ is {\em factorizable} and {\em ribbon}, with ribbon element 
\begin{equation}\label{ribbon}
V:=\sum_{u \in \B} S(\delta_u)\bowtie u.
\end{equation} 

\subsection{The representation category of $D(G)$}\label{sec:The representation category of D(G)} 
Fix a finite group scheme $G$ over $\mathbf{k}$ as in \S\ref{sec:Finite group schemes}. Let ${\rm ad_{\ell}}$ be the {\em left} adjoint action of $G$ on itself, defined by
$${\rm ad_{\ell}}(u)(\tilde{u})=u_1\tilde{u}S(u_2);\quad \forall u,\tilde{u}\in \mathbf{k}[G].$$ 
Let ${\rm C}:=G/{\rm ad_{\ell}}$ be the finite quotient scheme of {\em conjugacy orbits} in $G$ \cite{GS, GS2}, and let ${\rm p}:G\twoheadrightarrow {\rm C}:=G/{\rm ad_{\ell}}$ be the quotient scheme morphism \cite[Chapter 5]{J}. Then ${\rm p}^{\sharp}:\mathscr{O}({\rm C})\xrightarrow{1:1} \mathscr{O}(G)$ is an injective algebra map, and  
$$\im ({\rm p}^{\sharp})=\mathscr{O}(G)^{{\rm co}G}:=\{b\in \mathscr{O}(G)\mid {\rm ad_{\ell}}^{\sharp}(b)=1\ot b\}$$
is a subalgebra of $\mathscr{O}(G)$, so that $\mathscr{O}({\rm C})=\mathscr{O}(G)^{{\rm co}G}$ via ${\rm p}^{\sharp}$.

Given a closed point $g\in G(\mathbf{k})$, consider the scheme morphism
\begin{equation}\label{adg}
{\rm ad}_g:G\to G,\quad f\mapsto {\rm ad_{\ell}}(g,f)=fgf^{-1},
\end{equation}
and let 
\begin{equation}\label{Cg}
C_g:=\im({\rm ad}_g)\subset G
\end{equation}
be the {\em conjugacy orbit} of $g$ \cite{GS, GS2}. Then 
$$\mathbf{k}[C_g]:=\{{\rm ad_{\ell}}(u)(g)\mid u\in \mathbf{k}[G]\}$$
is a {\em subcoalgebra} of $\mathbf{k}[G]$, $C_g(\mathbf{k})\subset G(\mathbf{k})$ is the conjugacy class of $g$ in $G(\mathbf{k})$, and ${\rm C}(\mathbf{k})=\{C_g\mid g\in G(\mathbf{k})\}$. Also, let
\begin{equation}\label{Gg}
G_g:=\{f\in G\mid {\rm ad}_g(f)=g\}
\end{equation}
be the centralizer of $g$ in $G$; it is a subgroup scheme of $G$, so that 
$$\mathbf{k}[G_g]:=\{u\in \mathbf{k}[G]\mid ug=gu\}$$ 
is a {\em Hopf subalgebra} of $\mathbf{k}[G]$. The scheme morphism ${\rm ad}_g$ (\ref{adg}) factors through the finite quotient scheme $G/G_g$, and induces 
a canonical scheme isomorphism
\begin{equation}\label{iotag}
i_g:G/G_g\xrightarrow{\cong} C_g
\end{equation}
(see \cite[Chapter 5]{J} for more details). In other words, the map
$$i_g^{\sharp}:\mathscr{O}(C_g)\to \mathscr{O}(G/G_g)=\mathscr{O}(G)^{G_g}\subseteq \mathscr{O}(G),\quad \langle i_g^{\sharp}(c),u\rangle=\langle c,u_1gS(u_2)\rangle,$$
is an algebra isomorphism.

Recall that $\mathscr{O}(G/G_g)\subseteq \mathscr{O}(G)$ is a {\em left} coideal subalgebra; that is,
$$\Delta\left(\mathscr{O}(G/G_g)\right)\subset \mathscr{O}(G)\ot \mathscr{O}(G/G_g).$$
Let $q_g:=q_{G_g}:\mathscr{O}(G)\twoheadrightarrow \mathscr{O}(G_g)$ be as in (\ref{qL}), and
consider the Hopf-Galois extension
$$0\to \mathscr{O}(G/G_g)\hookrightarrow \mathscr{O}(G)\xrightarrow{q_g}\mathscr{O}(G_g)\to 0,$$
where we view $\mathscr{O}(G)$ as a right $\mathscr{O}(G_g)$-comodule via
$$\rho_g:\mathscr{O}(G)\to \mathscr{O}(G)\ot \mathscr{O}(G_g),\quad b\mapsto b_1\ot q_g(b_2),$$
so that $\mathscr{O}(G/G_g)=\mathscr{O}(G)^{\text{co} \mathscr{O}(G_g)}=\{b\in \mathscr{O}(G)\mid b_1\ot q_g(b_2)=b\ot 1\}\subseteq \mathscr{O}(G)$.
Let
$$\mu_g:=\mu_{G_g}:\mathscr{O}(G_g)\xrightarrow{1:1}\mathscr{O}(G)$$
be a section for $q_g$ as in (\ref{muL}), and let
$$\alpha_g:\mathscr{O}(G)\twoheadrightarrow\mathscr{O}(G/G_g),\quad b\mapsto b_1\mu_g^{-1}(q_g(b_2)),$$
be a retraction for the inclusion map $\mathscr{O}(G/G_g)\subseteq \mathscr{O}(G)$;  
it is $\mathscr{O}(G/G_g)$-linear. Recall that 
$\mu_g$ is a convolution invertible $\mathscr{O}(G/G_g)$-colinear map; that is, $\varepsilon \mu_g = \varepsilon$ and 
\begin{equation}\label{mu-gcolinear}
\mu_g(d_1)\ot d_2=\mu_g(d)_1\ot q_g(\mu_g(d)_2);\quad\forall d\in \mathscr{O}(G_g),
\end{equation}
such that for every $d\in \mathscr{O}(G_g)$, we have
\begin{equation}\label{mu-greta-g}
q_g(\mu_g(d))=d\quad and \quad\alpha_g(\mu_g(d))=\varepsilon(d).
\end{equation}

For each $C\in {\rm C}(\mathbf{k})$ (\ref{Cg}), let $\pi_C:C\twoheadrightarrow C(\mathbf{k})$ and $\gamma_C:C(\mathbf{k})\xrightarrow{1:1} C$ be the scheme morphisms obtained from $\pi:G\twoheadrightarrow G(\mathbf{k})$ and $\gamma:G(\mathbf{k})\xrightarrow{1:1} G$ (\ref{sesG}) by restriction, 
and define the $\mathscr{O}(G)$-linear algebra maps 
\begin{gather*}
\chi_{C}:=\id\ot \gamma_{C}^{\sharp}:\mathscr{O}(G^{\circ})\ot\mathscr{O}(C)\twoheadrightarrow \mathscr{O}(G^{\circ})\ot \mathscr{O}(C(\mathbf{k})),\,\,\,{\rm and}\\
\nu_{C}:=\id\ot \pi_{C}^{\sharp}:\mathscr{O}(G^{\circ})\ot \mathscr{O}(C(\mathbf{k}))\xrightarrow{1:1}\mathscr{O}(G^{\circ})\ot\mathscr{O}(C).
\end{gather*}
(See \cite{GS,GS2} for more details.)

Now let $\mathscr{Z}(G):=\Rep(D(G))$ be the category of finite dimensional $\mathbf{k}$-representations of $D(G)$; it is a finite {\em non-degenerate ribbon} braided tensor category. Recall \cite{Ge} that there is a canonical tensor equivalence 
\begin{equation}\label{the center is gequiv}
{\rm Coh}(G)^{G}\simeq \mathscr{Z}(G),
\end{equation}
where ${\rm Coh}(G)^{G}$ is the category of $G$-equivariant sheaves on $G$ with respect to the {\em right} conjugation action of $G$ on itself (\ref{right adjoint action}). That is, an object of ${\rm Coh}(G)^{G}$ is an object $X\in {\rm Coh}(G)$ equipped with a right $\mathscr{O}(G)$-comodule structure $\rho:X\to X\ot \mathscr{O}(G)$, such that $\rho(a\cdot x)={\rm ad_r}^{\sharp}(a)\cdot \rho(x)$. In particular, $\mathscr{Z}(G)$ is a group scheme-theoretical category \cite{Ge}.

For each $C\in {\rm C}(\mathbf{k})$, let $\mathscr{Z}(G)_{C}\subset \mathscr{Z}(G)$   
be the full Abelian subcategory consisting of all objects annihilated by the defining ideal $\mathscr{I}(C)\subset \mathscr{O}(G)$ of $C$, and let $\overline{\mathscr{Z}(G)_{C}}$ be the {\em Serre closure} of $\mathscr{Z}(G)_{C}$ inside $\mathscr{Z}(G)$; that is, $\overline{\mathscr{Z}(G)_{C}}$ is the full Abelian subcategory of $\mathscr{Z}(G)$ consisting of all objects whose composition factors belong to $\mathscr{Z}(G)_{C}$.

\begin{theorem}\cite[Theorem 8.3]{GS}\label{maings}
The following hold:
\begin{enumerate}
\item
For each $C_g\in {\rm C}(\mathbf{k})$, with representative $g\in C_g(\mathbf{k})$, there is an equivalence of Abelian categories
\begin{equation}\label{funfc}
\mathbf{F}_{C_g}:{\rm Corep}(\mathscr{O}(G_g))\xrightarrow{\simeq} \mathscr{Z}(G)_{C_g},\,\,\,(M,\rho_M)\mapsto \left(\mathscr{O}(C_g)\otimes M,\rho_M^g\right),
\end{equation}
where $\rho_M:M\to \mathscr{O}(G_g)\otimes M$, $m\mapsto \sum m^{(-1)}\ot m^{(0)}$, and 
\begin{equation*}
\begin{split}
& \rho_M^g:\mathscr{O}(C_g)\otimes M\to \mathscr{O}(C_g)\otimes M\otimes \mathscr{O}(G),\\
& \rho_M^g(c\otimes m)= 
\left(i_g^{-1}\right)^{\sharp}\alpha_{g}\left(i_g^{\sharp}(c)_1\mu_g\left(m^{(-1)}\right)_1 \right)\otimes
m^{(0)}\otimes i_g^{\sharp}(c)_2\mu_g\left(m^{(-1)}\right)_2.
\end{split}
\end{equation*}
Here, $\mathscr{O}(G)$ acts on the first factor of $\mathscr{O}(C_g)\otimes M$ via the surjective algebra map $q_{C_g}:\mathscr{O}(G)\twoheadrightarrow \mathscr{O}(C_g)$. 

In particular, the composition functor 
$$\Rep(G)\simeq {\rm Corep}(\mathscr{O}(G))\xrightarrow{\mathbf{F}_{1}}\mathscr{Z}(G)_1\hookrightarrow \mathscr{Z}(G)$$
coincides with the canonical embedding $\Rep(G)\hookrightarrow \mathscr{Z}(G)$ of braided  categories.
\item
For each $C_g\in {\rm C}(\mathbf{k})$, with representative $g\in C_g(\mathbf{k})$, there is a bijection between equivalence classes of simple objects $(M,\rho_M)\in {\rm Corep}(\mathscr{O}(G_g))$ and simple objects of $\mathscr{Z}(G)_{C_g}$, assigning $(M,\rho_M)$ to $\mathbf{F}_{C_g}(M,\rho_M)$. Moreover, we have a direct sum decomposition of Abelian categories
$$\mathscr{Z}(G)=\bigoplus_{C_g\in {\rm C}(k)}\overline{\mathscr{Z}(G)_{C_g}},$$
and
$\overline{\mathscr{Z}(G)_{1}}\subseteq \mathscr{Z}(G)$ is a braided subcategory. In particular, if $G$ is connected then the simples of $\mathscr{Z}(G)$ are precisely those of $\Rep(G)\simeq {\rm Corep}(\mathscr{O}(G))$.
\item
For each simple $(M,\rho_M)\in {\rm Corep}(\mathscr{O}(G_g))$, with projective cover $P_{G_g}(M,\rho_M)$, we have
$$P_{\mathscr{Z}(G)}\left(\mathbf{F}_{C_g}(M,\rho_M)\right)\cong 
\left(\mathscr{O}(G^{\circ})\ot\mathscr{O}(C_g(\mathbf{k}))\ot P_{G_g}(M,\rho_M),R_M^g\right),$$
where $\mathscr{O}(G)$ acts diagonally on the first two factors, and   
$$R_{M}^g:=\left(\chi_{C_g}\ot\id^{\ot 2}\right)\left(\id_{\mathscr{O}(G^{\circ})}\ot\rho_{P_{G_g}(M,\rho_M)}^g\right)\left(\nu_{C_g}\ot\id\right).$$
\item
For each $(M,\rho_M)\in {\rm Corep}(\mathscr{O}(G_g))$, we have 
$${\rm FPdim}\left(\mathbf{F}_{C_g}(M,\rho_M)\right)=|C_g|{\rm dim}_{\mathbf{k}}\left(M\right),\quad\text{and}$$ 
$${\rm FPdim}\left(P_{\mathscr{Z}(G)}\left(\mathbf{F}_{C_g}(M,\rho_M)\right)\right)=[G:G_g(\mathbf{k})]{\rm dim}_{\mathbf{k}}\left(P_{G_g}(M,\rho_M)\right).\qquad\qquad\qquad\qquad\qed$$ 
\end{enumerate}
\end{theorem}

\section{Hopf algebra quotients of $D(G)$}\label{sec:Hopf algebra quotients of D(G)}

Fix a finite group scheme $G$ over $\mathbf{k}$ as in \S\ref{sec:Finite group schemes}.

\subsection{Construction of Hopf quotients of $D(G)$}\label{sec:Construction of quotients}
Let $K$ be a {\em normal} subgroup scheme of $G$. Let $q_K: \mathscr{O}(G) \twoheadrightarrow \mathscr{O}(K)$ be the corresponding surjective Hopf algebra map, and 
\begin{equation}\label{mu}
\mu_K:\mathscr{O}(K)\xrightarrow{1:1} \mathscr{O}(G)
\end{equation}
be a section for $q_{K}$ as in (\ref{qL})-(\ref{muL}). 
Since $K$ is normal in $G$, the Hopf subalgebra $\mathscr{O}(G/K)$ of $\mathscr{O}(G)$ is stable under the left coadjoint  $\mathbf{k}[G]$-action (\ref{left coadjoint action}), so we obtain a well defined action $*$ of $\mathbf{k}[G]$ on $\mathscr{O}(K)$, given by
\begin{equation}\label{themeasureofog}
\mathbf{k}[G] \otimes \mathscr{O}(K) \to \mathscr{O}(K),\quad u\otimes a\mapsto u * a:= q_K(u \coadj \mu_K(a)),
\end{equation}
which is independent of the choice of $\mu_K$. Indeed, if $\mu_K':\mathscr{O}(K)\to \mathscr{O}(G)$ is another section of $q_K$, then $\mu_K'(a)-\mu_K(a)\in \mathscr{O}(G/K)^+\mathscr{O}(G)$ for every $a\in \mathscr{O}(K)$.  
Since the Hopf ideal $\mathscr{O}(G/K)^+\mathscr{O}(G)$ is stable under the left coadjoint $\mathbf{k}[G]$-action (\ref{left coadjoint action}), we have 
$u \coadj (\mu_K'(a)-\mu_K(a)) \in \mathscr{O}(G/K)^+\mathscr{O}(G)$ for all $u \in \mathbf{k}[G]$. Hence,
$$q_K (u \coadj \mu_K'(a)) = q_K (u \coadj \mu_K(a)),$$ 
so $*$ (\ref{themeasureofog}) is independent of the choice of $\mu_K$, as claimed.

In particular, we have
\begin{equation}\label{piyaction}
u * q_K(b):= q_K(u \coadj b);\quad \forall u\in \mathbf{k}[G],\, b\in \mathscr{O}(G).
\end{equation}

Suppose now that $H$ is another {\em normal} subgroup scheme of $G$, so that we have an exact sequence of Hopf algebras
\begin{equation*}
\mathbf{k}\xrightarrow{\iota_H} \mathbf{k}[H]\to \mathbf{k}[G]\xrightarrow{\pi_H} \mathbf{k}[G/H] \to \mathbf{k},
\end{equation*}
and choose a section $\gamma_H$ as in (\ref{gamma}). Assume further that $H$ and $K$ {\em centralize} each other; that is, $vw=wv$ for every $v\in \mathbf{k}[H]$ and $w \in \mathbf{k}[K]$. Equivalently, 
$${\rm ad_{r}}(v)(w)=S(v_1)wv_2 = \varepsilon(v)w;\quad\forall v\in \mathbf{k}[H],\,w \in \mathbf{k}[K].$$ 
Then for any $v \in \mathbf{k}[H]$ and $b \in \mathscr{O}(G)$, we have for every $w \in \mathbf{k}[K]$,
\begin{eqnarray*}
\lefteqn{\gen{q_K(v \coadj b), w} = \gen{v \coadj b, \iota_K(w)}}\\
& = & \gen{b, {\rm ad_{r}}(v)(\iota_K(w))} = \gen{b, \varepsilon(v)\iota_K(w)}=\gen{\varepsilon(v)q_K(b),w},
\end{eqnarray*} 
so $q_K(v \coadj b) = \varepsilon(v)q_K(b)$.  
Thus, (\ref{themeasureofog}) induces a well defined $\mathbf{k}$-linear map $\cdot$, given by 
\begin{equation}\label{themeasure}
\mathbf{k}[G/H] \otimes \mathscr{O}(K) \to \mathscr{O}(K),\quad x \otimes a\mapsto x\cdot a:=\gamma_H(x)*a =q_K(\gamma_H(x)\coadj \mu_K(a)),
\end{equation}
which is independent of the choice of $\gamma_H$ (as well as $\mu_K$). Indeed, if $\gamma_H':\mathbf{k}[G/H]\to \mathbf{k}[G]$ is another section of $\pi_H$, then for each $x\in \mathbf{k}[G/H]$, we have $\gamma_H'(x)-\gamma_H(x)\in \mathbf{k}[G]\mathbf{k}[H]^+$. But for any $uv$, $u\in \mathbf{k}[G]$ and $v\in \mathbf{k}[H]^+$, we have that 
$$(uv)*a=u*(v*a)=u*(q_K(v \coadj \mu_K(a)))=u*(\varepsilon(v)q_K(\mu_K(a)))=0$$
for every $a\in \mathscr{O}(K)$.  
Therefore, for every $a\in \mathscr{O}(K)$, 
$$\gamma_H'(x)*a = (\gamma_H(x)+\gamma_H'(x)-\gamma_H(x))*a = \gamma_H(x)*a,$$ 
so $\cdot$ (\ref{themeasure}) is independent of the choice of $\gamma_H$, as claimed.

In particular, we have
\begin{equation}\label{piyactiononq}
\pi_H(u)\cdot a=u*a;\quad \forall u\in \mathbf{k}[G],\, a\in \mathscr{O}(K),
\end{equation}
and by (\ref{piyaction}), we have
\begin{equation}\label{piyactiononq0}
\pi_H(u)\cdot q_K(b)=u*q_K(b)= q_K(u \coadj b);\quad \forall  u\in \mathbf{k}[G],\, b\in \mathscr{O}(G).
\end{equation}

\begin{lemma}\label{measure}
The Hopf algebra $\mathbf{k}[G/H]$ acts and measures $\mathscr{O}(K)$ via (\ref{themeasure}); that is, for every $x,y\in \mathbf{k}[G/H]$ and $a,\tilde{a}\in \mathscr{O}(K)$, we have
$$x\cdot(y\cdot a)=(xy)\cdot a,\quad x \cdot 1=\varepsilon(x),\quad \text{and}\quad x \cdot (a\tilde{a})=(x_1 \cdot a)(x_2 \cdot \tilde{a}).$$
\end{lemma}

\begin{proof}
First we have $(xy)\cdot a=\gamma_H(xy)*a$, and 
$$x\cdot(y\cdot a)=x\cdot (\gamma_H(y)*a)=\gamma_H(x)*(\gamma_H(y)*a)=(\gamma_H(x)\gamma_H(y))*a,$$
so the first equation follows from $\gamma_H(xy)-\gamma_H(x)\gamma_H(y)\in \mathbf{k}[G]\mathbf{k}[H]^+$.

Second, since $\mu_K(1)=1$ and $\varepsilon(\gamma_H(x))=\varepsilon(x)$, we have for every $x\in \mathbf{k}[G/H]$,
$$x \cdot 1= q_K(\gamma_H(x) \coadj 1) = q_K(\varepsilon(\gamma_H(x))) = \varepsilon(x),$$
as claimed. 

Finally, since we have 
\begin{equation}\label{useful}
\begin{split}
& q_K(\mu_K(a\tilde{a}))=a\tilde{a}=q_K(\mu_K(a)\mu_K(\tilde{a})),\, \text{and}\\
& (\pi_H\otimes \pi_H)(\gamma_H(x)_1\otimes \gamma_H(x)_2)=\Delta(x)=(\pi_H\otimes \pi_H)(\gamma_H(x_1)\otimes \gamma_H(x_2)),
\end{split}
\end{equation}
it follows that
\begin{eqnarray*}
\lefteqn{x \cdot (a\tilde{a})= q_K(\gamma_H(x) \coadj \mu_K(a\tilde{a}))=q_K(\gamma_H(x) \coadj \mu_K(a)\mu_K(\tilde{a}))}\\
& = & q_K(\gamma_H(x)_1 \coadj \mu_K(a))(\gamma_H(x)_2 \coadj \mu_K(\tilde{a}))\\
& = & q_K(\gamma_H(x)_1 \coadj \mu_K(a))q_K(\gamma_H(x)_2 \coadj \mu_K(\tilde{a}))\\
& = & q_K(\gamma_H(x_1) \coadj \mu_K(a))q_K(\gamma_H(x_2) \coadj \mu_K(\tilde{a}))= (x_1 \cdot a)(x_2 \cdot \tilde{a}),
\end{eqnarray*}
as claimed, where we used (\ref{useful}) in the second and fifth equations.
\end{proof}

Now suppose in addition that we have a Hopf algebra map
\begin{equation}\label{B}
B : \mathbf{k}[H] \to \mathscr{O}(K),
\end{equation}
which is {\em $G$-equivariant}, in the sense that 
\begin{equation}\label{Binv}
\pi_H(u)\cdot B(v)=u*B(v)=B({\rm ad_{\ell}}(u)(v));\quad \forall u \in \mathbf{k}[G],\,v \in \mathbf{k}[H].
\end{equation}

\begin{remark}\label{bindgiso}
(1) The $G$-equivariant Hopf algebra map $B$ (\ref{B}) can be viewed as a {\em $G$-invariant bicharacter} $B:H\times K\to \mathbb{G}_m$, in the sense that the Hopf algebra pairing
\begin{equation}\label{Bpairing}
\mathbf{k}[H]\ot \mathbf{k}[K]\to \mathbf{k},\quad v\ot w\mapsto \left\langle B(v),w\right\rangle,
\end{equation}
is $G$-invariant; that is,
\begin{equation}\label{G-invariant}
\left\langle B({\rm ad_{\ell}}(u_1)(v)),{\rm ad_{\ell}}(u_2)(w)\right\rangle=\varepsilon(u)\left\langle B(v),w\right\rangle
\end{equation}
or equivalently, 
$$\langle B(\ad_{\ell}(u)(v)), w\rangle = \langle B(v), \ad_{\ell}(S(u))(w) \rangle$$
for every $u \in \mathbf{k}[G]$, $v \in \mathbf{k}[H]$ and $w \in \mathbf{k}[K]$.

(2) Since the image of $B$ (\ref{B}) is a cocommutative Hopf subalgebra of $\mathscr{O}(K)$, there exists a normal subgroup scheme $L\subseteq K$, such that $K/L$ is commutative, so that we have $\im(B)=\mathscr{O}(K/L)=\mathbf{k}[\left(K/L\right)^{\vee}]$, where $\left(K/L\right)^{\vee}$ is the {\em Cartier dual} of $K/L$ (see, e.g. \cite{GS}). Thus, if ${\rm Ker}(B)=\mathbf{k}[N]^+\mathbf{k}[H]$, $N\subseteq H$ a normal subgroup scheme, then $H/N$ is a commutative group scheme, and $B$ induces a $G$-equivariant group scheme isomorphism $B:H/N\xrightarrow{\cong}\left(K/L\right)^{\vee}$. \qed
\end{remark}

Now let $\eta_H$, $\overline{\sigma}$ and $\overline{\tau}$ be as in (\ref{etafromgamma}), (\ref{defnolsigma}) and (\ref{defnoltau}), respectively, and define the $\mathbf{k}$-linear maps 
$$\sigma:=B\overline{\sigma}\quad\text{and}\quad\tau:=B^{\ot 2}\overline{\tau};$$
that is,
\begin{equation}\label{defnsigma}
\sigma : \mathbf{k}[G/H] \otimes \mathbf{k}[G/H] \to \mathscr{O}(K),\quad 
\sigma(x,y)=B\left(\gamma_H(x_1)\gamma_H(y_1)\gamma_H^{-1}(x_2y_2)\right),
\end{equation}
and
\begin{equation}\label{defntau}
\begin{split}
& \tau : \mathbf{k}[G/H] \to \mathscr{O}(K)\otimes \mathscr{O}(K),\\ 
& \tau(x)=(B\ot B)
\left\{
\eta^{-1}_H(\gamma_H(x)_1)_1\eta_H(\gamma_H(x)_2)\otimes \eta^{-1}_H(\gamma_H(x)_1)_2\eta_H(\gamma_H(x)_3)
\right\},
\end{split}
\end{equation}
and write $\tau(x)=\tau(x)^1\otimes \tau(x)^2$. Note that $\tau(x)^1\otimes \tau(x)^2=\tau(x)^2\otimes \tau(x)^1$, by cocommutativity of $\mathbf{k}[G]$.

\begin{remark}\label{trivial}
(1) By the discussion preceding Remark \ref{newtriv}, $\sigma$ and $\tau$ are well-defined.

(2) By Remark \ref{newtriv}, if $\gamma_H$ is an algebra map then $\sigma$ (\ref{defnsigma}) is trivial, and if 
$\eta_H$ (\ref{etafromgamma}) is a coalgebra map, equivalently if $\gamma_H$ (\ref{gamma}) is a coalgebra map (e.g. if $H=1$, or $G=G(\mathbf{k})$ is constant), then $\tau$ (\ref{defntau}) is trivial.

(3) If $B=1$ then both $\sigma$ (\ref{defnsigma}) and $\tau$ (\ref{defntau}) are trivial. \qed
\end{remark}

\begin{lemma}\label{newlemma}
Set $\pi:=\pi_H$, and $\eta:=\eta_H$. 
For every $u,\tilde{u}\in\mathbf{k}[G]$, the following hold:
\begin{enumerate}
\item
$\sigma(\pi(u),\pi(\tilde{u}))=B\left\{\eta(\gamma(\pi(u))\gamma(\pi(\tilde{u})))\right\}$.
\item
$\tau(\pi(u))=(B\ot B)\left\{\eta^{-1}(u_1)_1\eta(u_2)\ot \eta^{-1}(u_1)_2\eta(u_3)\right\}$.
\end{enumerate}
\end{lemma}

\begin{proof}
(1) For every $x,y\in \mathbf{k}[G/H]$, we have by (\ref{gammacolinear}),
\begin{eqnarray*}
\lefteqn{\eta(\gamma(x)\gamma(y))=\gamma(x)_1\gamma(y)_1\gamma^{-1}(\pi(\gamma(x)_2\gamma(y)_2))}\\
& = & \gamma(x)_1\gamma(y)_1\gamma^{-1}(\pi(\gamma(x)_2)\pi(\gamma(y)_2))=
\gamma(x_1)\gamma(y_1)\gamma^{-1}(x_2y_2),
\end{eqnarray*}
so the claim follows.

(2) Since $\gamma(\pi(u))=\eta^{-1}(u_1)u_2$ by Lemma \ref{gcphgh}(4), we have
$$\gamma(\pi(u))_1\ot \gamma(\pi(u))_2\ot \gamma(\pi(u))_3=\eta^{-1}(u_1)_1u_2\ot \eta^{-1}(u_1)_2u_3\ot \eta^{-1}(u_1)_3u_4.$$
Thus, we have
\begin{eqnarray*}
\lefteqn{\tau(\pi(u))= B^{\ot 2}
\{
\eta^{-1}(\gamma(\pi(u))_1)_1\eta(\gamma(\pi(u))_2)\otimes \eta^{-1}(\gamma(\pi(u))_1)_2\eta(\gamma(\pi(u))_3)
\}}\\
& = & B^{\ot 2}
\{
\eta^{-1}(\eta^{-1}(u_1)_1u_2)_1\eta(\eta^{-1}(u_1)_2u_3)\otimes \eta^{-1}(\eta^{-1}(u_1)_1u_2)_2\eta(\eta^{-1}(u_1)_3u_4)
\}\\
& = & B^{\ot 2}
\{\eta^{-1}(u_2)_1S(\eta^{-1}(u_1)_1)
\eta^{-1}(u_1)_2\eta(u_3)\ot
\eta^{-1}(u_2)_2S(\eta^{-1}(u_1)_3)
\eta^{-1}(u_1)_4\eta(u_4)\}\\
& = & B^{\ot 2}\{\eta^{-1}(u_2)_1\varepsilon(\eta^{-1}(u_1)_1)
\eta(u_3)\ot\eta^{-1}(u_2)_2S(\eta^{-1}(u_1)_2)
\eta^{-1}(u_1)_3\eta(u_4)\}\\
& = & B^{\ot 2}\{\eta^{-1}(u_2)_1\eta(u_3)\ot
\eta^{-1}(u_2)_2S(\eta^{-1}(u_1)_1)
\eta^{-1}(u_1)_2\eta(u_4)\}\\
& = & B^{\ot 2}\{\eta^{-1}(u_2)_1\eta(u_3)\ot
\eta^{-1}(u_2)_2\varepsilon(\eta^{-1}(u_1))\eta(u_4)\}\\
& = & B^{\ot 2}\{\eta^{-1}(u_2)_1\eta(u_3)\ot
\eta^{-1}(u_2)_2\varepsilon(u_1)\eta(u_4)\}=  B^{\ot 2}
\{\eta^{-1}(u_1)_1\eta(u_2)\ot\eta^{-1}(u_1)_2\eta(u_3)\},
\end{eqnarray*}
as claimed.
\end{proof}

\begin{theorem}\label{quotientha}
Let $K,H\subseteq G$ be commuting normal subgroup schemes of $G$, and let $\cdot$, $B$, $\sigma$, and $\tau$ be as in (\ref{themeasure}), (\ref{B}),   (\ref{defnsigma}) and (\ref{defntau}), respectively.
Then there exists a pair $\left(D(K,H,B),\theta\right)$, such that the following hold:
\begin{enumerate}
\item
 $$D(K,H,B)=\mathscr{O}(K)^{{\rm cop}} \#^{\tau}_{\sigma} \mathbf{k}[G/H]$$
is a Hopf algebra with multiplication given by  
\begin{equation*}
(a \# x)(\tilde{a} \# \tilde{x}) =  a(x_1 \cdot \tilde{a}) \sigma(x_2, \tilde{x}_1) \# x_3 \tilde{x}_2,
\end{equation*}
comultiplication map $\delta$ given by  
\begin{equation*}
\delta(a\# x)=\left(a_2\tau(x_1)^1\# x_2\right)\otimes \left(a_1\tau(x_1)^2\# x_3\right),
\end{equation*}
and antipode map $S$ given by  
\begin{equation*}
S(a\# x)=(1\# S(x))(S(a)\# 1).
\end{equation*}
\item
$$\theta:D(G)\to D(K,H,B),\quad b\bowtie u\mapsto q_K(b)B(\eta_H(u_1))\# \pi_H(u_2),$$ is a surjective Hopf algebra map whose kernel is the ideal generated by $\mathscr{O}\left(G/K\right)^+$ and the vector subspace
$\left\{\mu_K(B(v))\bowtie 1 - 1\bowtie v\mid v\in \mathbf{k}[H]\right\}$.
\end{enumerate}
\end{theorem}

\begin{proof}
It suffices to show that $\theta$ is a surjective map, such that
\begin{equation}\label{p1}
\theta\left((b\bowtie u)(\tilde{b}\bowtie \tilde{u})\right)=\theta(b\bowtie u)\theta(\tilde{b}\bowtie \tilde{u}),\quad\text{and}
\end{equation} 
\begin{equation}\label{p2}
\theta(b_2\bowtie u_1) \otimes \theta(b_1\bowtie u_2)=\delta(\theta(b\bowtie u))
\end{equation}
for every $b,\tilde{b} \in \mathscr{O}(G)$ and $u,\tilde{u} \in \mathbf{k}[G]$.
 
Set $\pi:=\pi_H$, $\gamma:=\gamma_H$, $\eta:=\eta_H$, $\mu:=\mu_K$, and $q:=q_K$.
 
First we verify (\ref{p1}). On one hand, we have 
\begin{eqnarray*}
\lefteqn{\theta\left((b\bowtie u)(\tilde{b}\bowtie \tilde{u})\right) = \theta(b(u_1 \coadj \tilde{b})\bowtie u_2\tilde{u}) = q(b(u_1 \coadj \tilde{b}))B(\eta(u_2\tilde{u}_1)) \# \pi(u_3 \tilde{u}_2)}\\
& = & q(b(u_1 \coadj \tilde{b}))B\left\{\eta(u_2)\gamma(\pi(u_3))_1\eta(\tilde{u}_1)S(\gamma(\pi(u_3))_2)\eta\left\{\gamma(\pi(u_3))_3\gamma(\pi(\tilde{u}_2))\right\}\right\} \# \pi(u_4 \tilde{u}_3)\\
& = & q(b(u_1 \coadj \tilde{b}))B\left\{\eta(u_2)\left\{{\rm ad_{\ell}}(\gamma(\pi(u_3))_1)(\eta(\tilde{u}_1))\right\}\eta\left\{\gamma(\pi(u_3))_2\gamma(\pi(\tilde{u}_2))\right\}\right\} \# \pi(u_4 \tilde{u}_3)\\
& = & q(b(u_1 \coadj \tilde{b}))B(\eta(u_2))\{\gamma(\pi(u_3))_1 * B(\eta(\tilde{u}_1))\}B\left\{\eta\left\{\gamma(\pi(u_3))_2\gamma(\pi(\tilde{u}_2))\right\}\right\} \# \pi(u_4 \tilde{u}_3)\\
& = & q(b(u_1 \coadj \tilde{b}))B(\eta(u_2))\{\pi(\gamma(\pi(u_3))_1)\cdot B(\eta(\tilde{u}_1))\}B\left\{\eta\left\{\gamma(\pi(u_3))_2\gamma(\pi(\tilde{u}_2))\right\}\right\} \# \pi(u_4 \tilde{u}_3)\\
& = & q(b(u_1 \coadj \tilde{b}))B(\eta(u_2))\{\pi(u_3)_1\cdot B(\eta(\tilde{u}_1))\}B\left\{\eta\left\{\gamma(\pi(u_3)_2)\gamma(\pi(\tilde{u}_2))\right\}\right\} \# \pi(u_4 \tilde{u}_3)\\
& = & q(b(u_1 \coadj \tilde{b}))B(\eta(u_2))\{\pi(u_3)\cdot B(\eta(\tilde{u}_1))\}B\left\{\eta\left\{\gamma(\pi(u_4))\gamma(\pi(\tilde{u}_2))\right\}\right\} \# \pi(u_5 \tilde{u}_3),
\end{eqnarray*} 
where we used Lemma \ref{gcphgh}(3) in the third equality, (\ref{Binv}) in the fifth one, and (\ref{gammacolinear}) in the seventh one. 

On the other hand, we have
\begin{eqnarray*}
\lefteqn{\theta(b\bowtie u)\theta(\tilde{b}\bowtie \tilde{u})}\\ 
& = & \{q(b)B(\eta(u_1))\#\pi(u_2)\}\{q(\tilde{b})B(\eta(\tilde{u}_1))\#\pi(\tilde{u}_2) \}\\ 
& = & q(b)B(\eta(u_1))\{\pi(u_2)\cdot (q(\tilde{b})B(\eta(\tilde{u}_1)))\}\sigma(\pi(u_3),\pi(\tilde{u}_2))\#\pi(u_4\tilde{u}_3)\\ 
& = & q(b)B(\eta(u_1))\{(\pi(u_2)\cdot q(\tilde{b}))(\pi(u_3)\cdot B(\eta(\tilde{u}_1)))\}\sigma(\pi(u_4),\pi(\tilde{u}_2))\#\pi(u_5\tilde{u}_3)\\ 
& = & q(b(u_1 \coadj \tilde{b}))B(\eta(u_2))\{\pi(u_3)\cdot B(\eta(\tilde{u}_1))\}\sigma(\pi(u_4),\pi(\tilde{u}_2))\#\pi(u_5\tilde{u}_3)\\ 
& = & q(b(u_1 \coadj \tilde{b}))B(\eta(u_2))\{\pi(u_3)\cdot B(\eta(\tilde{u}_1))\}B\left\{\eta\left\{\gamma(\pi(u_4))\gamma(\pi(\tilde{u}_2))\right\}\right\} \# \pi(u_5 \tilde{u}_3),
\end{eqnarray*}
as desired, where we used Lemma \ref{measure} in the third equality, (\ref{piyactiononq}) and cocommutativity of $\mathbf{k}[G]$ in the fourth one, and Lemma \ref{newlemma}(1) in the last one.

Next we verify (\ref{p2}). On one hand, we have
\begin{eqnarray*}
\lefteqn{(\theta\ot \theta)\Delta(b\bowtie u)=\theta(b_2\bowtie u_1) \otimes \theta(b_1\bowtie u_2)}\\
& = & (q(b_2)B(\eta(u_1))\#\pi(u_2))\otimes (q(b_1)B(\eta(u_3))\#\pi(u_4))).
\end{eqnarray*}
On the other hand, we have
\begin{eqnarray*}
\lefteqn{\delta(\theta(b\bowtie u))=\delta(q(b)B(\eta(u_1))\#\pi(u_2))}\\
& = & \left(q(b_2)B(\eta(u_1)_2)\tau(\pi(u_2))^1\# \pi(u_3)\right)
\otimes \left(q(b_1)B(\eta(u_1)_1)\tau(\pi(u_2))^2\# \pi(u_4)\right)\\
& = & \left(q(b_2)B(\eta(u_1)_2\eta^{-1}(u_2)_1\eta(u_3))\# \pi(u_4)\right)
\otimes \left(q(b_1)B(\eta(u_1)_1\eta^{-1}(u_2)_2\eta(u_5))\# \pi(u_6)\right)\\
& = & \left(q(b_2)B(\eta(u_1)_1\eta^{-1}(u_2)_1\eta(u_3))\# \pi(u_4)\right)
\otimes \left(q(b_1)B(\eta(u_1)_2\eta^{-1}(u_2)_2\eta(u_5))\# \pi(u_6)\right)\\
& = & \left(q(b_2)B((\eta(u_1)\eta^{-1}(u_2))_1\eta(u_3))\# \pi(u_4)\right)
\otimes \left(q(b_1)B((\eta(u_1)\eta^{-1}(u_2))_2\eta(u_5))\# \pi(u_6)\right)\\
& = & \left(q(b_2)B(\eta(u_1))\# \pi(u_2)\right)
\otimes \left(q(b_1)B(\eta(u_3))\# \pi(u_4)\right),
\end{eqnarray*}
as desired, where we used Lemma \ref{newlemma}(2) in the second equation and the cocommutativity of $\mathbf{k}[G]$ in the last equation.
  
Finally we verify that $\theta$ is surjective. Clearly, $\theta(b\bowtie 1) = q(b) \# 1$ for any $b \in \mathscr{O}(G)$. Also for any $u \in \mathbf{k}[G]$, using the cocommutativity of $\mathbf{k}[G]$, we have
\begin{eqnarray*}
\lefteqn{\theta\left(\mu(B(\eta^{-1}(u_1)))\bowtie u_2\right) = q\left(\mu(B(\eta^{-1}(u_1)))\right)B(\eta(u_2))\# \pi(u_3)}\\
& = & B(\eta^{-1}(u_1))B(\eta(u_2))\# \pi(u_3)=B(\eta^{-1}(u_1)\eta(u_2))\# \pi(u_3)\\
& = & B(\varepsilon(u_1)) \# \pi(u_2) = 1 \# \pi(u). 
\end{eqnarray*}
Thus, $\theta$ is surjective, as claimed.

Finally, let $I$ be the ideal of $D(G)$ generated by $\mathscr{O}(G/K)^+$ and the subspace 
$$\left\{\mu(B(v))\bowtie 1 - 1\bowtie v\mid v\in \mathbf{k}[H]\right\}.$$ 
Since $\theta|_{\mathscr{O}(K)} = q$, we clearly have $\mathscr{O}(G/K)^+ \subseteq \ker \theta$. Moreover, for every $v \in \mathbf{k}[H]$, we have $\theta(\mu(B(v)) \bowtie 1) = B(v) \# 1 = \theta(1 \bowtie v)$. Thus, $I \subseteq \ker \theta$.
	
On the other hand, quotienting $D(G)$ by the ideal generated by $\mathscr{O}(G/K)^+$ identifies the $\mathscr O(G)$--factor with $\mathscr{O}(K)$, hence every class in $D(G)/I$ has a representative of the form $a\bowtie u$ with $a\in\mathscr O(K)$ and $u\in \mathbf{k}[G]$. Moreover, since $\mu(B(v))\bowtie 1 = 1\bowtie v$ for every $v\in \mathbf{k}[H]$,  the class of $a\bowtie u$ in $D(G)/I$ depends only on the image of $u$ in the quotient Hopf algebra $\mathbf{k}[G/H]$. It follows that $D(G)/I$ is spanned by the elements represented by $a\bowtie \gamma(x)$ with $a\in\mathscr{O}(K)$ and $x \in \mathbf{k}[G/H]$. Consequently, $\dim \left(D(G)/I\right) \le \dim \left(\mathscr{O}(K)\right)\dim \left(\mathbf{k}[G/H]\right)$. Hence, $\dim (I) \ge \dim (\ker \theta)$. 

Thus, $I=\ker \theta$, as claimed.
\end{proof}

\begin{remark}\label{trivialtaubar}
It follows from Theorem \ref{quotientha} that $\sigma$ (\ref{defnsigma}) is a {\em cocycle}, $\tau$ (\ref{defntau}) is a {\em co-cocycle}, and $\mathbf{k}[G]$ {\em co-measures} $\mathscr{O}(K)$ trivially (see e.g., \cite{AD,M} for the necessary definitions). \qed
\end{remark}

\begin{proposition}\label{genproprs}
For each Hopf algebra $D(K,H,B)$, the following hold:
\begin{enumerate}
\item
$D(K,H,B)$ is involutive, with $\dim_{\mathbf{k}}(D(K,H,B))=|K|[G:H]$.
\item
The map 
$$\mathscr{O}(K)^{{\rm cop}}\xrightarrow{1:1} D(K,H,B),\quad a\mapsto a\# 1,$$ 
is an injective Hopf algebra map.
\item
The map 
$$D(K,H,B)\twoheadrightarrow \mathbf{k}[G/H],\quad a\# x\mapsto \varepsilon(a)x,$$
is a surjective Hopf algebra map whose kernel is the ideal generated by $\mathscr{O}(K)^+$.
\item
If $\eta_H$ (\ref{etafromgamma}) is a coalgebra map, then $D(K,H,B)=\mathscr{O}(K)^{{\rm cop}}\otimes \mathbf{k}[G/H]$ is a tensor product coalgebra.
\item
If $B=1$, then $D(K,H):=D(K,H,1)=\mathscr{O}(K)^{{\rm cop}}\# \mathbf{k}[G/H]$ is a smash product algebra, and $D(K,H)=\mathscr{O}(K)^{{\rm cop}}\otimes \mathbf{k}[G/H]$ is a tensor product coalgebra.
\item
The map 
\begin{equation*}
\bar{B}:=B^*\circ S:\mathbf{k}[K]\to \mathscr{O}(H)
\end{equation*}
is a $G$-equivariant Hopf algebra map, so we have the corresponding Hopf algebra quotient $D(H,K,\bar{B})$ of $D(G)$. 
\item
If $K$ is commutative then $D(K,H,B)=\mathbf{k}[\widetilde{G}]$, where $\widetilde{G}$ is the group scheme extension of $G/H$ by $K^{\vee}$ corresponding to 
$$\sigma\in Z^2(G/H,K^{\vee})\quad\text{and}\quad\tau\in Z^1(G/H,(K^2)^{\vee}),$$
where $G/H$ acts trivially on $K^{\vee}$ and $(K^2)^{\vee}$.
\end{enumerate}
\end{proposition}

\begin{proof}
(1)-(2) These are straightforward.

(3) This is straightforward noting that $\varepsilon(x\cdot a)=\varepsilon(x)\varepsilon(a)$, $\varepsilon(\sigma(x,\tilde{x}))=\varepsilon(x)\varepsilon(\tilde{x})$, and $(\varepsilon\otimes \varepsilon)(\tau(x))=\varepsilon(x)(1\otimes 1)$.

(4)-(5) These follow in a straightforward manner using Remark \ref{trivial}.

(6) This follows since $\mathbf{k}[H]$ is cocommutative.

(7) This follows since $D(K,H,B)$ is cocommutative when $K$ is commutative.
\end{proof}

\subsection{Classification of Hopf quotient pairs of $D(G)$}\label{sec:Classification of quotients}
A {\em Hopf quotient pair of $D(G)$} is a pair $(D,\varphi)$ where $\varphi:D(G)\twoheadrightarrow D$ is a surjective Hopf algebra map. We will say that two Hopf quotient pairs $(D_1,\varphi_1)$ and $(D_2,\varphi_2)$ are {\em equivalent} if there is a Hopf algebra isomorphism $f:D_1\xrightarrow{\cong} D_2$, such that $\varphi_2=f\varphi_1$.

\begin{lemma}\label{uniquenesstriples}
Two Hopf quotient pairs $\left(D(K,H,B),\theta\right)$ and $\left(D\left(K',H',B'\right),\theta'\right)$ as constructed in Theorem \ref{quotientha} are equivalent if and only if $(K,H,B)=(K',H',B')$. 
\end{lemma}

\begin{proof}
Suppose that $f:D(K,H,B)\xrightarrow{\cong} D\left(K',H',B'\right)$ is a Hopf algebra isomorphism, such that $\theta'=f\theta$. In particular, 
$\theta'_{\mid \mathscr{O}(G)}=(f\theta)_{\mid \mathscr{O}(G)}$; that is, $q_{K'}=(f_{\mid \mathscr{O}(K)})\circ q_K$, or equivalently, $\iota_{K'}=\iota_K\circ\widetilde{(f_{\mid \mathscr{O}(K)})}$, where $\widetilde{(f_{\mid \mathscr{O}(K)})}:K'\twoheadrightarrow K$ is the corresponding surjective group scheme morphism. Thus, $K=K'$ and $f_{\mid \mathscr{O}(K)}:\mathscr{O}(K)\to \mathscr{O}(K)$ is the identity map.

Also, for every $v\in \mathbf{k}[H]$, we have $\eta_H(v)=v$ and $\pi_H(v)=\varepsilon(v)$, so we have that 
$\theta(1\bowtie v)=B(v)\#1$. Applying $\varepsilon\otimes\mathrm{id}$ to 
$$f(B(v)\#1)=f\theta(1\bowtie v)=\theta'(1\bowtie v)=B'(\eta_{H'}(v_1))\#'\pi_{H'}(v_2)$$ 
then yields $\pi_{H'}(v)=\varepsilon(v)$. Hence, $\mathbf{k}[H]^+\subseteq\ker(\pi_{H'})=\mathbf{k}[G]\mathbf{k}[H']^+$. By symmetry, we obtain $\ker(\pi_H)=\ker(\pi_{H'})$, and therefore $H=H'$. 

Finally, it follows from the above and Theorem \ref{quotientha}(2) that  
$$B'=\theta'_{\mid \mathbf{k}[H]}=(f\circ\theta)_{\mid \mathbf{k}[H]}=f_{\mid \mathscr{O}(K)}\circ \theta_{\mid \mathbf{k}[H]}=\theta_{\mid \mathbf{k}[H]}=B$$
since $f_{\mid \mathscr{O}(K)}:\mathscr{O}(K)\to \mathscr{O}(K)$ is the identity map.
\end{proof}

\begin{theorem}\label{uniqueness}
The assignment $(K,H,B)\mapsto (D(K,H,B),\theta)$ constructed in Theorem \ref{quotientha} determines a bijection between the set of triples $(K,H,B)$ and the set of equivalence classes of Hopf quotient pairs of $D(G)$. 
\end{theorem}

\begin{proof}
By Theorem \ref{quotientha} and Lemma \ref{uniquenesstriples}, it remains to prove that for any Hopf quotient pair $(D,\varphi)$ of $D(G)$, there exist a Hopf quotient pair $(D(K,H,B),\theta)$ of $D(G)$, and a Hopf algebra isomorphism $\bar{\varphi}:D(K,H,B)\xrightarrow{\cong}D$, such that $\bar{\varphi} \theta = \varphi$.

So, let $\varphi : D(G) \twoheadrightarrow D$ be a surjective Hopf algebra map. Since the image of $\varphi_{\mid \mathscr{O}(G)}$ is a Hopf algebra quotient of $\mathscr{O}(G)$, we may assume that 
$\varphi(\mathscr{O}(G))=\mathscr{O}(K)\subseteq D$, and $\varphi_{\mid \mathscr{O}(G)}=q_K:\mathscr{O}(G)\twoheadrightarrow \mathscr{O}(K)$ for some subgroup scheme $K\subseteq G$. Moreover, since $\mathscr{O}(G)\subseteq D(G)$ is a normal Hopf subalgebra, $\mathscr{O}(K)\subseteq D$ is a normal Hopf subalgebra.
\vspace{0.2cm}

\noindent$\mathbf{Claim\, 1.}$
$K\subseteq G$ is a normal subgroup scheme. 
\vspace{0.2cm}

\noindent{\em Proof.}
Fix $a\in \mathscr{O}(G/K)^+\mathscr{O}(G)$. Then $\varphi(a\bowtie 1)=q_K(a)=0$, hence  
$$0=\varphi\left((1\bowtie u_1)(a\bowtie 1)(1\bowtie S(u_2)\right) = \varphi\left((u\coadj a)\bowtie 1\right)=q_K(u\coadj a);$$
that is, $u\coadj a\in \mathscr{O}(G/K)^+\mathscr{O}(G)$ for every $u\in \mathbf{k}[G]$. It follows that ${\rm coad}$ (\ref{coadog}) maps $\mathscr{O}(G/K)^+\mathscr{O}(G)$ to $\mathscr{O}(G/K)^+\mathscr{O}(G)\ot \mathscr{O}(G)$; that is, $K$ is normal in $G$. \qed

Next consider the exact sequence of Hopf algebras
$$\mathbf{k}\to \mathscr{O}(K) \to D \to D/(\mathscr{O}(K)^+D)\to \mathbf{k}.$$ 
Since $D(G)/(\mathscr{O}(G)^+D(G))\cong \mathbf{k}[G]$ (\ref{qkg}), it follows that the Hopf algebra $D/(\mathscr{O}(K)^+D)$ is a Hopf quotient of $\mathbf{k}[G]$, so we may assume that 
$$D/(\mathscr{O}(K)^+D)=\mathbf{k}[G/H]$$
for some {\em normal} subgroup scheme $H$ of $G$, so that we have a commutative diagram with exact rows
\begin{equation}\label{commdiagram}
\begin{tikzcd}
\mathbf{k}\arrow[r] & \mathscr{O}(G) \arrow[r] \arrow[d, "q_K"] & D(G) \arrow[r, "\varepsilon\ot \id"] \arrow[d, "\varphi"] & {\mathbf{k}[G]} \arrow[r] \arrow[d, "\pi_H"] & \mathbf{k}\\
\mathbf{k}\arrow[r] & \mathscr{O}(K) \arrow[r] & D \arrow[r, "p"] & \mathbf{k}[G/H] \arrow[r] & \mathbf{k}.
\end{tikzcd}
\end{equation}
In particular, by \cite{AD, S}, we may assume that $D=\mathscr{O}(K) \#^{\tilde{\tau}}_{\tilde{\sigma}} \mathbf{k}[G/H]$ as Hopf algebras, so that $D=\mathscr{O}(K)\ot \mathbf{k}[G/H]$ as vector spaces, and we can view $\mathscr{O}(K)$ as a Hopf subalgebra of $D$ in the obvious way. To avoid confusion, we will write $a\widetilde{\#}x$ to denote an element of $D$.
\vspace{0.2cm}

\noindent$\mathbf{Claim\, 2.}$
For every $v\in \mathbf{k}[H]$, $\varphi(1\bowtie v)\in \mathscr{O}(K)$. Moreover, the Hopf algebra map 
$$B:=\varphi_{\mid \mathbf{k}[H]} : \mathbf{k}[H] \to \mathscr{O}(K),\quad v\mapsto \varphi(1\bowtie v),$$
is $G$-equivariant as in (\ref{G-invariant}).
\vspace{0.2cm}

\noindent{\em Proof.}
By (\ref{commdiagram}), for every $v\in \mathbf{k}[H]$, we have 
\begin{eqnarray*}
\lefteqn{(\id\otimes p)\Delta_D(\varphi(1\bowtie v))=(\id\otimes p)(\varphi(1\bowtie v_1)\otimes \varphi(1\bowtie v_2))}\\ 
& = & \varphi(1\bowtie v_1)\otimes p(\varphi(1\bowtie v_2))=\varphi(1\bowtie v_1)\otimes \pi_H(v_2)\\
& = &\varphi(1\bowtie v_1)\otimes \varepsilon(v_2)= \varphi(1\bowtie v)\otimes 1,
\end{eqnarray*}
which implies that $\varphi(1\bowtie v)\in D^{{\rm co}\mathbf{k}[G/H]}=\mathscr{O}(K)$. 

Furthermore, for every $v \in \mathbf{k}[H]$ and $u \in \mathbf{k}[G]$, we have
\begin{eqnarray*}
\lefteqn{B({\rm ad_{\ell}}(u)(v))= \varphi(1\bowtie u_1vS(u_2))}\\
& = & \varphi((1\bowtie u_1)(1\bowtie v)(1\bowtie S(u_2))) =\varphi(1\bowtie u_1)\varphi(1\bowtie v)\varphi(1\bowtie S(u_2))\\
& = & \varphi(1\bowtie u_1)\varphi(\mu_K(B(v))\bowtie 1)\varphi(1\bowtie S(u_2))=\varphi((1\bowtie u_1)(\mu_K(B(v))\bowtie 1)(1\bowtie S(u_2)))\\
& = & \varphi((u_1\coadj \mu_K(B(v))\bowtie u_2)(1\bowtie S(u_3)))=\varphi((u_1\coadj \mu_K(B(v)))\bowtie u_2S(u_3))\\
& = & \varphi((u\coadj \mu_K(B(v)))\bowtie 1)=q_K(u\coadj \mu_K(B(v)))=u*B(v).
\end{eqnarray*}
Here $\mu_K$ is a section of $q_K$ as in (\ref{muL}). Thus, for every $w\in \mathbf{k}[K]$, we have
\begin{eqnarray*}
\lefteqn{\left\langle B({\rm ad_{\ell}}(u_1)(v)),{\rm ad_{\ell}}(u_2)(w)\right\rangle=\left\langle q_K\left(u_1\coadj \mu_K(B(v))\right),{\rm ad_{\ell}}(u_2)(w)\right\rangle}\\
& = & \left\langle q_K\left(S(u_2)\coadj \mu_K(q_K\left(u_1\coadj \mu_K(B(v))\right))\right),w\right\rangle\\
& = & \left\langle S(u_2)*q_K\left(u_1\coadj \mu_K(B(v))\right),w\right\rangle=
\left\langle S(u_2)*(u_1*B(v)),w\right\rangle\\
& = & \varepsilon(u)\left\langle B(v),w\right\rangle,
\end{eqnarray*}
as claimed.
\qed
\vspace{0.1cm}

\noindent$\mathbf{Claim\, 3.}$
The normal subgroup schemes $K,H\subseteq G$ centralize each other.
\vspace{0.2cm}

\noindent{\em Proof.}
We claim that $\mathbf{k}[H]$ acts trivially on $\mathscr{O}(K)$ via $\varphi$. Indeed, 
for every $v \in \mathbf{k}[H]$ and $b \in \mathscr{O}(G)$, we have
\begin{eqnarray*}
\lefteqn{\varphi(1\bowtie v)_1\varphi(b\bowtie 1)S(\varphi(1\bowtie v)_2) = \varphi(1\bowtie v_1)\varphi(b\bowtie 1)\varphi(1\bowtie S(v_2)) }\\ 
& = & \varphi(1\bowtie v_1)\varphi(1\bowtie S(v_2))\varphi(b\bowtie 1)= \varphi(1\bowtie v_1S(v_2))\varphi(b\bowtie 1)= \varepsilon(v) \varphi(b\bowtie 1),
\end{eqnarray*}
as claimed, where we applied Claim 2 to use the commutativity of $\mathscr{O}(K)$ in the second equality.

Now since $(\ad_{\ell}(1\bowtie v))(b\bowtie 1) = (v\coadj b)\bowtie 1$, it follows that 
$$q_K(v\coadj b) = \varepsilon(v) q_K(b); \quad\forall\,v\in\mathbf{k}[H],\ b\in\mathscr{O}(G).$$ 
Pairing with $w \in \mathbf{k}[K]$ and using (\ref{left coadjoint action2}), we obtain 
$$\langle b,{\rm ad_r}(v)(w)\rangle=\langle v\coadj b,w\rangle = \langle q_K(v\coadj b), w \rangle = \varepsilon(v)\langle q_K(b), w \rangle = \varepsilon(v) \langle b, w \rangle.$$ 
Since this holds for all $b \in \mathscr{O}(G)$, it follows that ${\rm ad_r}(v)(w) = \varepsilon(v) w$ for all $w \in \mathbf{k}[K]$; that is, $\mathbf{k}[H]$ and $\mathbf{k}[K]$ commute. \qed

Set $\pi:=\pi_H$, $\gamma:=\gamma_H$, $\eta:=\eta_H$, and $q:=q_K$.
\vspace{0.2cm}

\noindent$\mathbf{Claim\, 4.}$
For every $u\in \mathbf{k}[G]$, we have $\varphi(1\bowtie u)=B(\eta(u_1))\widetilde{\#} \pi(u_2)$.
\vspace{0.2cm}

\noindent{\em Proof.} 
First note that the commutativity of (\ref{commdiagram}) implies that for every $x\in \mathbf{k}[G/H]$,
$$\varphi(1\bowtie \gamma(x))=1\widetilde{\#} x.$$
Now since by Lemma \ref{gcphgh}(2), $u=\eta(u_1)(\gamma(\pi(u_2)))$ for every $u\in \mathbf{k}[G]$, it follows that 
\begin{eqnarray*}
\lefteqn{\varphi(1\bowtie u)=\varphi(1\bowtie \eta(u_1)(\gamma(\pi(u_2))))= \varphi(1\bowtie \eta(u_1))\varphi(1\bowtie \gamma(\pi(u_2)))}\\
& = & (B(\eta(u_1))\widetilde{\#} 1)(1\widetilde{\#} \pi(u_2))= B(\eta(u_1))\widetilde{\#} \pi(u_2),
\end{eqnarray*}
as claimed. \qed
\vspace{0.2cm}

It now follows from Claims 2 and 4 that we have
\begin{equation}\label{varphiis}
\varphi(b\bowtie u)=\varphi(b\bowtie 1)\varphi(1\bowtie u)=q(b)B(\eta(u_1))\widetilde{\#} \pi(u_2);\quad \forall\, b\bowtie u\in D(G).
\end{equation}

Finally, let $(D(K,H,B),\theta)$ be the Hopf quotient pair of $D(G)$ corresponding, by Theorem \ref{quotientha} and Claims 2-3, to the triple $(K,H,B)$, and consider the identity map 
$$\bar{\varphi}:=\id_{\mathscr{O}(K)\ot \mathbf{k}[G/H]} : D(K,H,B) \to D,\quad a \# x\mapsto  a\widetilde{\#}x.$$
Then $\bar{\varphi} \theta = \varphi$ by (\ref{varphiis}), so $\theta = \varphi$, and  
$\bar{\varphi}$ is a Hopf algebra isomorphism, so $(D,\varphi)$ is equivalent to $(D(K,H,B),\theta)$.
\end{proof}

\begin{corollary}\label{lattice}
For each Hopf quotient pair $(D(K,H,B),\theta)$ of $D(G)$ and surjective Hopf algebra homomorphism $\varphi:D(K,H,B)\twoheadrightarrow D$, the Hopf quotient pair $(D,\varphi\theta)$ is equivalent to a Hopf quotient pair $(D(K',H',B'),\theta')$ for some normal subgroup schemes $K'\subseteq K$ and $H\subseteq H'$ of $G$, such that 
$$\iota^{\sharp}_{K',K}\circ B =B'\circ \iota_{H,H'}:\mathbf{k}[H]\to \mathscr{O}(K').$$
\end{corollary}

\begin{proof}
Since $(D,\varphi\theta)$ is a Hopf quotient pair of $D(G)$, it follows from Theorem \ref{uniqueness} that there exists a Hopf quotient pair $(D(K',H',B'),\theta')$ together with a Hopf algebra isomorphism $f:D(K',H',B')\xrightarrow{\cong} D$, such that $\varphi\theta=f\theta'$. In particular, we have
$$(\varphi\theta)|_{\mathscr{O}(G)}=\left(\varphi|_{\mathscr{O}(K)}\right)\circ q_K.$$ On the other hand, since $\varphi\theta=f\theta'$ and $\theta'|_{\mathscr{O}(G)}=q_{K'}$, we have $(\varphi\theta)|_{\mathscr{O}(G)}=\left(f|_{\mathscr{O}(K')}\right)\circ q_{K'}$. Thus, $\varphi q_K=f q_{K'}$, so $f^{-1}\circ\left(\varphi|_{\mathscr{O}(K)}\right):\mathscr{O}(K)\twoheadrightarrow  \mathscr{O}(K')$ is a surjective Hopf algebra map, so $K'\subseteq K$ is a subgroup scheme.

Next, for every $v\in \mathbf{k}[H]^+$, we have
$$\varphi(B(v)\# 1)=\varphi\theta(1\bowtie v)=f\theta'(1\bowtie v)=f\left(B'(\eta_{H'}(v_1))\#\pi_{H'}(v_2)\right),$$
thus, we have
$$B'(\eta_{H'}(v_1))\#\pi_{H'}(v_2)=f^{-1}\left(\varphi(B(v)\# 1)\right)\in \mathscr{O}(K')^+\# 1.$$
Applying $\varepsilon\ot\id$ to both sides of the last equation yields $\pi_{H'}(v)=0$, which implies that $v\in \mathbf{k}[H']^+\mathbf{k}[G]$. Thus, $\mathbf{k}[H]^+\mathbf{k}[G]\subseteq \mathbf{k}[H']^+\mathbf{k}[G]$; that is, 
$H\subseteq H'$, as desired.

Finally, since $\varphi\theta=f\theta'$, and Theorem \ref{quotientha}(2) gives $\theta|_{\mathbf{k}[H]}=B$ and $\theta'|_{\mathbf{k}[H']}=B'$, functoriality implies that we have $\iota^{\sharp}_{K',K}\circ B =B'\circ \iota_{H,H'}:\mathbf{k}[H]\to \mathscr{O}(K')$, as desired.
\end{proof}

\begin{corollary}\label{lattice2}
Let $(D(K,H,B),\theta)$ and $(D\left(K',H',B'\right),\theta')$ be two Hopf quotient pairs of $D(G)$, such that $K\subseteq K'$ and $H'\subseteq H$ are normal subgroup schemes of $G$, and
$\mathbf{q}\circ B' =B\circ \iota_{H',H}:\mathbf{k}[H']\to \mathscr{O}(K)$, where $\mathbf{q}:=\iota^{\sharp}_{K,K'}:\mathscr{O}(K')\twoheadrightarrow \mathscr{O}(K)$ is the canonical surjective Hopf algebra map. 
Then there is a unique surjective Hopf algebra map
$\varphi:D\left(K',H',B'\right)\twoheadrightarrow D(K,H,B)$, such that $\theta =\varphi\theta'$.
\end{corollary}

\begin{proof}
Since $K\subseteq K'$, we have $\mathscr{O}(G/K')\subseteq \mathscr{O}(G/K)$ as Hopf subalgebras of $\mathscr{O}(G)$, so $\mathscr{O}(G/K')^+\subseteq \mathscr{O}(G/K)^+$. Next fix $v\in \mathbf{k}[H']$ and consider $$r_v:=\mu_{K'}(B'(v))\bowtie 1-1\bowtie v\ \in D(G).$$	
We claim that $r_v\in \ker(\theta)$. First note that $$\theta(1\bowtie v)=B(\eta_H(v_1))\#\pi_H(v_2)=B(v_1)\#\varepsilon(v_2)=B(v)\# 1.$$ On the other hand, $$\theta(\mu_{K'}(B'(v))\bowtie 1)=q_K(\mu_{K'}(B'(v)))\# 1.$$ Using $\mathbf{q}\,q_{K'}=q_K$ and $q_{K'}\mu_{K'}={\rm id}_{\mathscr{O}(K')}$, we get
$$q_K(\mu_{K'}(B'(v)))=\mathbf{q}(q_{K'}(\mu_{K'}(B'(v))))=\mathbf{q}(B'(v)).$$ By assumption, $\mathbf{q}\circ B'=B\circ \iota_{H',H}$, and viewing $v\in \mathbf{k}[H']$ inside $\mathbf{k}[H]$ via $\iota_{H',H}$, this gives $\mathbf{q}(B'(v))=B(v)$. Hence,	
$$\theta(\mu_{K'}(B'(v))\bowtie 1)=B(v)\# 1=\theta(1\bowtie v),$$	
so $\theta(r_v)=0$, as claimed.
	
Thus, by Theorem \ref{quotientha}(3), we obtain $\ker(\theta') \subseteq \ker({\theta})$, so there is a unique surjective Hopf algebra map $\varphi:D(K',H',B')\twoheadrightarrow D(K,H,B)$, such that $\theta=\varphi\theta'$, as claimed.
\end{proof}

Fix two Hopf quotient pairs $(D(K,H,B),\theta)$ and $(D\left(K',H',B'\right),\theta')$ of $D(G)$, and  
define the Hopf algebra map
\begin{equation}\label{phibbprime}
\begin{split}
& \beta _{B,B'}:\mathbf{k}[K\cap K']\to \mathscr{O}(H\cap H')\\
& w\mapsto \left(\iota^{\sharp}_{H\cap H',H}B^*\iota_{K\cap K',K}(w_1)\right)\left(\iota^{\sharp}_{H\cap H',H'}\overline{B'}\iota_{K\cap K',K'}(w_2)\right);
\end{split}
\end{equation}
that is, $\beta_{B,B'} =\left(\iota^{\sharp}_{H\cap H',H}B^*\iota_{K\cap K',K}\right)\star\left(\iota^{\sharp}_{H\cap H',H'}\overline{B'}\iota_{K\cap K',K'}\right)$. 

Let $L\subseteq K\cap K'$ be the subgroup scheme, such that 
$${\rm Ker}(\beta_{B,B'})=\mathbf{k}[L]^+\mathbf{k}[K\cap K'],$$ 
and define the $G$-equivariant Hopf algebra map
\begin{equation}\label{psibbprime}
\begin{split}
& \mathbf{B}_{B,B'}:\mathbf{k}[L]\to \mathscr{O}(HH'),\\
& w\mapsto \left(\mu_{H,HH'}B^*\iota_{L,K}(w_1)\right)\left(\mu_{H',HH'}\overline{B'}\iota_{L,K'}(w_2)\right);
\end{split}
\end{equation}
that is, $\mathbf{B}_{B,B'}=\left(\mu_{H,HH'}B^*\iota_{L,K}\right)\star\left(\mu_{H',HH'}\overline{B'}\iota_{L,K'}\right)$.
Set 
\begin{equation}\label{Bint}
\mathbf{B}:=\mathbf{B}_{B,B'}^*:\mathbf{k}[HH']\to \mathscr{O}(L),
\end{equation} 
and let $\left(D\left(L,HH',\mathbf{B}\right),\Theta\right)$ be the corresponding Hopf quotient pair of $D(G)$.

\begin{proposition}\label{intersection}
The following hold:
\begin{enumerate}
\item
There exist surjective Hopf algebra maps $\varphi:D(K,H,B)\twoheadrightarrow D\left(L,HH',\mathbf{B}\right)$ and $\varphi':D\left(K',H',B'\right)\twoheadrightarrow D\left(L,HH',\mathbf{B}\right)$, such that $\varphi\theta=\Theta =\varphi'\theta'$.
\item
The Hopf quotient pair $\left(D\left(L,HH',\mathbf{B}\right),\Theta\right)$ is the maximal one having the properties in (1).
\end{enumerate}
\end{proposition}

\begin{proof}
(1) This follows from Corollary \ref{lattice2}.

(2) Suppose that 
$$\phi:D(K,H,B)\twoheadrightarrow D\left(K'',H'',B''\right),\quad\phi':D\left(K',H',B'\right)\twoheadrightarrow D\left(K'',H'',B''\right)$$
are surjective Hopf algebra maps, such that $\phi\theta=\theta'' =\phi'\theta'$. Then by Corollary \ref{lattice}, $K''\subseteq K\cap K'$ and $HH'\subseteq H''$ are normal subgroup schemes of $G$,
$$\iota^{\sharp}_{K'',K}\circ B=B''\circ \iota_{H,H''}:\mathbf{k}[H]\to \mathscr{O}(K''),\quad\text{and}$$
$$\iota^{\sharp}_{K'',K'}\circ B' =B''\circ \iota_{H',H''}:\mathbf{k}[H']\to \mathscr{O}(K'').$$

We claim that $K''\subseteq L$. Indeed, for every $w\in \mathbf{k}[K'']$ and $v\in \mathbf{k}[H\cap H']$, we have
\begin{eqnarray*}
\lefteqn{\langle \beta_{B,B'}(\iota_{K'',K\cap K'}(w)),v\rangle}\\
& = & \langle \iota^{\sharp}_{H\cap H',H}B^*\iota_{K\cap K',K}(\iota_{K'',K\cap K'}(w_1)),v_1\rangle 
\langle \iota^{\sharp}_{H\cap H',H'}\overline{B'}\iota_{K\cap K',K'}(\iota_{K'',K\cap K'}(w_2)),v_2\rangle\\
& = & \langle \iota^{\sharp}_{H\cap H',H}B^*\iota_{K'',K}(w_1),v_1\rangle 
\langle \iota^{\sharp}_{H\cap H',H'}\overline{B'}\iota_{K'',K'}(w_2),v_2\rangle\\
& = & \langle w_1,\iota^{\sharp}_{K'',K}B\iota_{H\cap H',H}(v_1)\rangle 
\langle S(w_2),\iota^{\sharp}_{K'',K'}B'\iota_{H\cap H',H'}(v_2)\rangle\\
& = & \langle w_1,B''\iota_{H,H''}\iota_{H\cap H',H}(v_1)\rangle 
\langle S(w_2),B''\iota_{H',H''}\iota_{H\cap H',H'}(v_2)\rangle\\
& = & \langle w_1,B''\iota_{H\cap H',H''}(v_1)\rangle 
\langle S(w_2),B''\iota_{H\cap H',H''}(v_2)\rangle\\
& = & \langle w_1S(w_2),B''\iota_{H\cap H',H''}(v)\rangle =\langle \varepsilon(w),B''\iota_{H\cap H',H''}(v)\rangle
=\langle \varepsilon(w),v\rangle.
\end{eqnarray*}
Thus, $\mathbf{k}[K'']^+\subseteq {\rm Ker}(\beta_{B,B'})=\mathbf{k}[L]^+\mathbf{k}[K\cap K']$, so $K''\subseteq L$, as claimed.

Next we claim that the surjective Hopf algebra map 
$$\Phi:D\left(L,HH',\mathbf{B}\right)\twoheadrightarrow D\left(K'',H'',B''\right)$$
given by Corollary \ref{lattice2}, satisfies $\Phi\varphi=\phi$ and $\Phi\varphi'=\phi'$. Indeed, by Corollary \ref{lattice2}, $\theta'' =\Phi\Theta$, and by assumption, $\varphi\theta=\Theta =\varphi'\theta'$ and $\phi\theta=\theta'' =\phi'\theta'$. Thus, we have
$$\phi\theta=\theta'' =\Phi\Theta=\Phi\varphi\theta,$$
so $\Phi\varphi=\phi$ since $\theta$ is surjective. Similarly, $\Phi\varphi'=\phi'$.
\end{proof}

\subsection{The quasitriangular ribbon structure on $D(K,H,B)$}\label{sec:Quasi-triangular ribbon structures}
Fix bases $\B_K\subseteq \B$  
for $\mathbf{k}[K]$ and $\mathbf{k}[G]$, respectively, and let $\{\delta_u\mid u\in \B\}$ be the dual basis for $\mathscr{O}(G)$. It is clear that $\{q_K(\delta_u)\mid u\in \B_K\}$ is the dual basis of $\B_K$ for $\mathscr{O}(K)$, and that for $u\in \B_K$, $\langle q_K(\delta_u),v\rangle\ne 0$ if and only if $v\in \B_K$.

\begin{corollary}\label{main0}
Each Hopf quotient $D(K,H,B)$ of $D(G)$ is quasitriangular ribbon, with $R$-matrix
$$R(K,H,B):=\sum_{u \in \B_K} (B(\eta_H(u_1))\#\pi_H(u_2))\otimes (q_K(\delta_u)\# 1),$$
and ribbon element
$$V(K,H,B):=\sum_{u \in \B_K} q_K(S(\delta_u))B(\eta_H(u_1))\#\pi_H(u_2).$$
\end{corollary}

\begin{proof}
By Theorem \ref{quotientha}(2), the element 
\begin{eqnarray*}
\lefteqn{R(K,H,B)=(\theta\otimes \theta)(R)}\\
& = & \sum_{u \in \B} \theta(1\bowtie u) \otimes \theta(\delta_u\bowtie 1)=\sum_{u \in \B_K}(B(\eta_H(u_1))\#\pi_H(u_2))\otimes (q_K(\delta_u)\# 1)
\end{eqnarray*}
is an $R$-matrix for $D(K,H,B)$, 
and
$$V(K,H,B)=\theta(V)=\sum_{u \in \B} \theta(S(\delta_u)\bowtie u)=\sum_{u \in \B_K}q_K(S(\delta_u))B(\eta_H(u_1))\#\pi_H(u_2)$$
is a ribbon element, as claimed.
\end{proof}

\section{Tensor subcategories of $\Rep(D(G))$}\label{sec:Tensor subcategories}

Fix a finite group scheme $G$ over $\mathbf{k}$ as in \S\ref{sec:Finite group schemes}.

\subsection{Construction and classification} 
Let $\mathscr{Z}(G):=\Rep(D(G))$ as in \S\ref{sec:The representation category of D(G)}. For each triple $(K,H,B)$ as in Theorem \ref{quotientha}, set
$$\mathscr{Z}(K,H,B):=\Rep(D(K,H,B)).$$
Then using the surjective Hopf algebra map $\theta:D(G)\twoheadrightarrow D(K,H,B)$, we can (and will) view $\mathscr{Z}(K,H,B)$ as a tensor subcategory of $\mathscr{Z}(G)$.

\begin{corollary}\label{main-1}
The tensor category $\mathscr{Z}(K,H,B)$ is ribbon braided, with Frobenius-Perron dimension $|K| [G : H]$.
\end{corollary}

\begin{proof}
This follows from Proposition \ref{genproprs}.
\end{proof}

\begin{theorem}\label{mainclass}
The assignment $(K,H,B)\mapsto \mathscr{Z}(K,H,B)$ determines a bijection between the set of triples $(K,H,B)$ and the set of tensor subcategories of $\mathscr{Z}(G)$.
\end{theorem}

\begin{proof}
It is well known that the set of tensor
subcategories of the representation category of a finite dimensional Hopf algebra is in bijection with the set of equivalence classes of its quotient pairs (see, e.g., \cite[Proposition 2]{BN}). Thus, the claim follows from Theorem \ref{uniqueness}.
\end{proof}

Theorem \ref{mainclass} yields a classification of certain finite braided group scheme-theoretical categories \cite{Ge}.

\begin{corollary}
Let $\mathscr{B}$ be any finite braided group scheme-theoretical category which admits a fiber functor to ${\rm Vect}$. Then there exists a finite group scheme $G$, and a triple $(K,H,B)$ as in Theorem \ref{quotientha}, such that $\mathscr{B}$ is braided tensor equivalent to $\mathscr{Z}(K,H,B)$. 
\end{corollary}

\begin{proof}
Since any finite braided group scheme-theoretical category $\mathscr{B}$ is a braided tensor subcategory 
of its center $\mathscr{Z}\left(\mathscr{B}\right)$, and by \cite{Ge} the latter is braided tensor equivalent to $\mathscr{Z}(G)$ for some finite group scheme $G$ (since $\mathscr{B}$ is tensor equivalent to a representation category of a Hopf algebra), the claim follows from Theorem \ref{mainclass}.
\end{proof}

\begin{corollary}\label{main-2}
The following hold:
\begin{enumerate}
\item
$\mathscr{Z}(K',H',B')\subseteq \mathscr{Z}(K,H,B)$ if and only if $K'\subseteq K$ and $H\subseteq H'$ are normal subgroup schemes of $G$, and $\iota^{\sharp}_{K',K}\circ B=B'\circ \iota_{H,H'}:\mathbf{k}[H]\to \mathscr{O}(K')$.
\item
$\mathscr{Z}(K,H,B)\cap \mathscr{Z}\left(K',H',B'\right)=\mathscr{Z}\left(L,HH',\mathbf{B}\right)$, where $\mathbf{B}$ is defined in (\ref{Bint}).
\end{enumerate}
\end{corollary}

\begin{proof}
(1) This follows from Corollary \ref{lattice} and Theorem \ref{mainclass}.

(2) This follows from Proposition \ref{intersection}.
\end{proof}

\subsection{The M\"uger centralizer of $\mathscr{Z}(K,H,B)$}
For each Hopf algebra $D(K,H,B)$, recall the Hopf algebra $D(H,K,\bar{B})$ from Proposition \ref{genproprs}(6).

\begin{theorem}\label{main2}
For each tensor subcategory $\mathscr{Z}(K,H,B)\subseteq \mathscr{Z}(G)$, the tensor subcategory $\mathscr{Z}(H,K,\bar{B})\subseteq \mathscr{Z}(G)$ is its M\"uger centralizer in $\mathscr{Z}(G)$. 
\end{theorem}

\begin{proof}
Fix $D:=D(K,H,B)$, and set $\bar{D}:=D(H,K,\bar{B})$. Also, let $\mathscr{Z}:=\mathscr{Z}(K,H,B)$, and $\bar{\mathscr{Z}}:=\mathscr{Z}(H,K,\bar{B})$. Let $\theta: D(G) \twoheadrightarrow D$ and $\bar{\theta}:D(G) \twoheadrightarrow \bar{D}$ be the Hopf algebra quotients constructed in Theorem \ref{quotientha}.

\vspace{0.2cm}
\noindent$\mathbf{Claim\,1.}$
$\mathscr{Z}$ and $\bar{\mathscr{Z}}$ centralize each other if and only if $(\theta \otimes \bar{\theta})(R_{21}R) = 1 \otimes 1$.
\vspace{0.2cm}

\noindent{\em Proof.} 
Recall that the square braiding in $\mathscr{Z}(G)$ is defined by multiplication by $R_{21}R$. 
If $\mathscr{Z}$ and $\bar{\mathscr{Z}}$ centralize each other, then $(\theta \otimes \bar{\theta})(R_{21}R)$ acts as the identity on $X \otimes Y$ for all $X\in \mathscr{Z}$ and $Y \in \bar{\mathscr{Z}}$. In particular, we can take the faithful representation $D\otimes \bar{D}$ of $D\otimes \bar{D}$, and get that $(\theta \otimes \bar{\theta}) (R_{21}R) = 1 \otimes 1$. 

Conversely, if $(\theta \otimes \bar{\theta}) (R_{21}R) = 1 \otimes 1$, then $R_{21}R$ acts by the identity on $X \otimes Y$ for all $X \in \mathscr{Z}$ and $Y \in \bar{\mathscr{Z}}$, so $\mathscr{Z}$ and $\bar{\mathscr{Z}}$ centralize each other, as claimed. \qed
\vspace{0.2cm}

Now fix a basis $\B$ for $\mathbf{k}[G]$, such that $\B$ contains bases $\B_H$ and $\B_K$ for $\mathbf{k}[H]$ and $\mathbf{k}[K]$, respectively, and let $\{\delta_u\mid u\in \B\}$ be the dual basis for $\mathscr{O}(G)$. By (\ref{rmatrix}), we have
$$R_{21}R= \sum_{u,\tilde{u} \in \B} (\delta_u\bowtie \tilde{u}) \otimes ((u_1\coadj \delta_{\tilde{u}})\bowtie u_2).$$ 
We claim that $(\theta \otimes \bar{\theta}) (R_{21}R)=1 \otimes 1$. Indeed, we have
\begin{eqnarray*}
\lefteqn{(\theta \otimes \bar{\theta}) (R_{21}R)=\sum_{u,\tilde{u} \in \B} \theta(\delta_u\bowtie \tilde{u}) \otimes \bar{\theta}((u_1\coadj \delta_{\tilde{u}})\bowtie u_2)}\\ 
& = & \sum_{u,\tilde{u} \in \B} \left(q_K(\delta_u) B(\eta_H(\tilde{u}_1)) \# \pi_H(\tilde{u}_2)\right) \otimes \left(q_H(u_1\coadj \delta_{\tilde{u}}) \bar{B}(\eta_K(u_2))\# \pi_K(u_3)\right)\\ 
& = & \sum_{u\in \B_K,\,\tilde{u}\in \B_H} \left(q_K(\delta_u) B(\eta_H(\tilde{u}_1)) \# \pi_H(\tilde{u}_2)\right) \otimes \left(q_H(u_1\coadj \delta_{\tilde{u}}) \bar{B}(\eta_K(u_2))\# \pi_K(u_3)\right)\\ 
& = & \sum_{u\in \B_K,\,\tilde{u}\in \B_H} \left(q_K(\delta_u) B(\eta_H(\tilde{u})) \# 1\right) \otimes \left(q_H(\delta_{\tilde{u}}) \bar{B}(\eta_K(u))\# 1\right)\\ 
& = & \sum_{u\in \B_K,\,\tilde{u}\in \B_H} \left(q_K(\delta_u) B(\tilde{u}) \# 1\right) \otimes \left(q_H(\delta_{\tilde{u}}) \bar{B}(u)\# 1\right).
\end{eqnarray*}
To show that 
$\sum_{u\in \B_K,\,\tilde{u}\in \B_H} \left(q_K(\delta_u) B(\tilde{u})\# 1\right)\otimes \left(q_H(\delta_{\tilde{u}}) \bar{B}(u)\# 1\right)=(1\# 1)\otimes (1\# 1)$, 
we compute for any $v\in \mathbf{k}[K]$ and $w\in \mathbf{k}[H]$,
\begin{eqnarray*}
\lefteqn{\sum_{u\in \B_K,\,\tilde{u}\in \B_H} \left\langle q_K(\delta_u) B(\tilde{u}),v\right\rangle \left\langle q_H(\delta_{\tilde{u}}) \bar{B}(u),w\right\rangle}\\
& = & \sum_{u\in \B_K,\,\tilde{u}\in \B_H} \left\langle q_K(\delta_u),v_1\right\rangle \left\langle B(\tilde{u}),v_2\right\rangle \left\langle q_H(\delta_{\tilde{u}}),w_1\right\rangle \left\langle \bar{B}(u),w_2\right\rangle\\
& = & \sum_{u\in \B_K,\,\tilde{u}\in \B_H} \left\langle q_K(\delta_u)\left\langle u,\bar{B}^*(w_2)\right\rangle,v_1\right\rangle  \left\langle q_H(\delta_{\tilde{u}})\left\langle \tilde{u},B^*(v_2)\right\rangle,w_1\right\rangle\\
& = & \left\langle \bar{B}^*(w_2),v_1\right\rangle \left\langle B^*(v_2),w_1\right\rangle 
=\left\langle \bar{B}^*(w_2),v_1\right\rangle \left\langle B(w_1),v_2\right\rangle=\left\langle \bar{B}^*(w_2)B(w_1),v\right\rangle\\
& = & \left\langle B(S(w_2))B(w_1),v\right\rangle =\left\langle B(S(w_2)w_1),v\right\rangle=\varepsilon(v)\varepsilon(w).
\end{eqnarray*}
Thus, $(\theta \otimes \bar{\theta}) (R_{21}R) = 1 \otimes 1$, as claimed.

Hence, it follows from Claim 1 that $\mathscr{Z}\subseteq \bar{\mathscr{Z}}'$, and since both categories have the same Frobenius-Perron dimension by Corollary \ref{main-1} and \cite[Theorem 4.9]{Sh}, an equality holds, as claimed.
\end{proof}

Recall (see, e.g., \cite{EGNO}) that a finite braided tensor category is called {\em non-degenerate} if its M\"uger center is trivial.

\begin{corollary}\label{main4}
For each tensor subcategory $\mathscr{Z}(K,H,B)\subseteq \mathscr{Z}(G)$, the following hold:
\begin{enumerate}
\item
$\left(D(K,H,B),R(K,H,B)\right)$ is triangular if and only if $\mathscr{Z}(K,H,B)$ is symmetric, if and only if, $K\subseteq H$ and $B\circ \iota_{K,H}=\iota_{K,H}^{\sharp}\circ \overline{B}:\mathbf{k}[K]\to \mathscr{O}(K)$. 

Furthermore, in this case $K$ is commutative, so $D(K,H,B)=\mathbf{k}[\widetilde{G}]$ as in Proposition \ref{genproprs}(7), and we have
$$R(K,H,B)=\sum_{u \in \B_K} (B(u)\#1)\otimes (q_K(\delta_u)\# 1)\in\mathbf{k}[K^{\vee}]^{\otimes 2}\quad\text{and}$$
$$V(K,H,B)=\sum_{u \in \B_K} q_K(S(\delta_u))B(u)\#1\in\mathbf{k}[K^{\vee}].$$
\item
$\left(D(K,H,B),R(K,H,B)\right)$ is factorizable if and only if $\mathscr{Z}(K,H,B)$ is non-degenerate, if and only if, $HK=G$ and 
$$\beta _{B,\overline{B}}=\left(\iota^{\sharp}_{K\cap H,K} B^* \iota_{K\cap H,H}\right)\star\left(\iota^{\sharp}_{K\cap H,H} B \iota_{K\cap H,K}\right):\mathbf{k}[K\cap H]\to \mathscr{O}(K\cap H)$$
is a Hopf algebra isomorphism. In particular, if $\mathscr{Z}(K,H,B)$ is non-degenerate then $K\cap H$ is a self dual central subgroup scheme of $G$.
\item
$\mathscr{Z}(K,H,B)$ is Lagrangian if and only if $K=H$ and $B=\bar{B}$. In particular, if $\mathscr{Z}(K,H,B)$ is Lagrangian then $K=H$ is commutative, so 
$$D(H,H,B)=\mathbf{k}[H^{\vee}]\#^{\tau}_{\sigma}\mathbf{k}[G/H]=\mathbf{k}[\widetilde{G}],$$ 
as in Proposition \ref{genproprs}(7).
\end{enumerate}
\end{corollary}

\begin{proof}
Fix $\mathscr{Z}:=\mathscr{Z}(K,H,B)\subseteq \mathscr{Z}(G)$.

(1) Since $\mathscr{Z}$ is symmetric if and only if $\mathscr{Z}\subseteq \mathscr{Z}'$, if and only if by Theorem \ref{main2}, $\mathscr{Z}(K,H,B)\subseteq \mathscr{Z}(H,K,\overline{B})$, the claim follows from Corollary \ref{main-2}(1).

(2) Since $\mathscr{Z}$ is non-degenerate if and only if $\mathscr{Z}'\cap \mathscr{Z}=\Vec$, if and only if by Theorem \ref{main2}, $\mathscr{Z}(K,H,B)\cap \mathscr{Z}(H,K,\overline{B})=\mathscr{Z}(1,G,1)$,  
and since by Corollary \ref{main-2}(2), $\mathscr{Z}(K,H,B)\cap \mathscr{Z}(H,K,\overline{B})=\mathscr{Z}(L,KH,\mathbf{B})$, we get that $\mathscr{Z}(K,H,B)$ is non-degenerate if and only if $G=KH$ and $L=1$, so $\beta _{B,\overline{B}}:\mathbf{k}[K\cap H]\to \mathscr{O}(K\cap H)$ is injective, hence a Hopf algebra isomorphism, as claimed.

(3) Since $\mathscr{Z}$ is Lagrangian if and only if $\mathscr{Z}'=\mathscr{Z}$, if and only if by Theorem \ref{main2}, $\mathscr{Z}(H,K,\overline{B})=\mathscr{Z}(K,H,B)$, the claim follows from Theorem \ref{mainclass}.
\end{proof}

\subsection{The simples and projectives of $\mathscr{Z}(K,H,B)$}
Retain the notation from \S\ref{sec:The representation category of D(G)}. 

Fix a triple $\left(K,H,B\right)$ as in Theorem \ref{quotientha}. Recall that $K\subseteq G$ is a normal subgroup scheme, so $C_g\subset K$ for every $g\in K(\mathbf{k})$, and let ${\rm C}_K(\mathbf{k}):=\{C_g\mid g\in K(\mathbf{k})\}\subseteq {\rm C}(\mathbf{k})$. 

Fix $C_g\in {\rm C}_K(\mathbf{k})$, with representative $g\in C_g(\mathbf{k})$, and let $(M,\rho_M)\in {\rm Corep}(\mathscr{O}(G_g))$. Recall that for $m\in M$, we write $\rho_M(m)=m^{(-1)}\ot m^{(0)}\in \mathscr{O}(G_g)\ot M$. Let 
\begin{equation*}
\begin{split}
& \rho_M^g:\mathscr{O}(C_g)\otimes M\to \mathscr{O}(C_g)\otimes M\otimes \mathscr{O}(G),\\
& \rho_M^g(c\otimes m)= 
\left(i_g^{-1}\right)^{\sharp}\alpha_{g}\left(i_g^{\sharp}(c)_1\mu_g\left(m^{(-1)}\right)_1 \right)\otimes
m^{(0)}\otimes i_g^{\sharp}(c)_2\mu_g\left(m^{(-1)}\right)_2,
\end{split}
\end{equation*}
and define $\mathbf{F}_{C_g}(M,\rho_M)=\left(\mathscr{O}(C_g)\otimes M,\rho_M^g\right)\in \mathscr{Z}(G)_{C_g}$. 
In particular, the action of $u\in\mathbf{k}[G]$ on $c\otimes m\in \mathscr{O}(C_g)\otimes M$ is given by
\begin{equation}\label{actGonCg}
\begin{split}
& u\cdot (c\ot m)=\left\langle u,i_g^{\sharp}(c)_2\mu_g\left(m^{(-1)}\right)_2\right\rangle \left(i_g^{-1}\right)^{\sharp}\alpha_{g}\left(i_g^{\sharp}(c)_1\mu_g\left(m^{(-1)}\right)_1 \right)\otimes
m^{(0)}.
\end{split}
\end{equation}
Note that since $g\in K$, $H\subseteq G$ is normal, and $H$ centralizes $K$, it follows that $H\subseteq G_g$ is a normal subgroup scheme, and $q_{G,H}=q_{G_g,H}\circ q_{G,G_g}$.

\begin{lemma}\label{expac}
For every $v\in\mathbf{k}[H]$ and $c\otimes m\in \mathscr{O}(C_g)\otimes M$, we have 
$$\iota_{H,G}(v)\cdot (c\ot m)=\left\langle \iota_{H,G_g}(v),m^{(-1)}\right\rangle
c\otimes m^{(0)}=c\otimes m\cdot \iota_{H,G_g}(v).$$
\end{lemma}

\begin{proof}
Since for every $c\in \mathscr{O}(C_g)$, $\Delta(i_g^{\sharp}(c))\in \mathscr{O}(G)\ot \mathscr{O}(G/G_g)$, it follows from (\ref{actGonCg}) that 
for every $v\in\mathbf{k}[H]\subseteq \mathbf{k}[G_g]$ and $c\otimes m\in \mathscr{O}(C_g)\otimes M$, we have
\begin{eqnarray*}
\lefteqn{\iota_{H,G}(v)\cdot (c\ot m)=\left\langle \iota_{H,G}(v),i_g^{\sharp}(c)_2\mu_g\left(m^{(-1)}\right)_2\right\rangle \left(i_g^{-1}\right)^{\sharp}\alpha_{g}\left(i_g^{\sharp}(c)_1\mu_g\left(m^{(-1)}\right)_1 \right)\otimes m^{(0)}}\\
& = & \left\langle \iota_{H,G}(v_1),i_g^{\sharp}(c)_2\right\rangle 
\left\langle \iota_{H,G}(v_2),\mu_g\left(m^{(-1)}\right)_2\right\rangle
\left(i_g^{-1}\right)^{\sharp}\alpha_{g}\left(i_g^{\sharp}(c)_1\mu_g\left(m^{(-1)}\right)_1 \right)\otimes m^{(0)}\\
& = & \varepsilon(i_g^{\sharp}(c)_2)\varepsilon(\iota_{H,G}(v_1)) 
\left\langle \iota_{H,G}(v_2),\mu_g\left(m^{(-1)}\right)_2\right\rangle
\left(i_g^{-1}\right)^{\sharp}\alpha_{g}\left(i_g^{\sharp}(c)_1\mu_g\left(m^{(-1)}\right)_1 \right)\otimes m^{(0)}\\
& = &  
\left\langle \iota_{H,G}(v),\mu_g\left(m^{(-1)}\right)_2\right\rangle
\left(i_g^{-1}\right)^{\sharp}\alpha_{g}\left(i_g^{\sharp}(c)\mu_g\left(m^{(-1)}\right)_1 \right)\otimes m^{(0)}\\
& = &  
\left\langle \iota_{H,G}(v),\mu_g\left(m^{(-1)}\right)_2\right\rangle
\left(i_g^{-1}\right)^{\sharp}\left(i_g^{\sharp}(c)\alpha_{g}\left(\mu_g\left(m^{(-1)}\right)_1 \right)\right)\otimes m^{(0)}\\
& = &  
\left\langle \iota_{H,G}(v),\mu_g\left(m^{(-1)}\right)_2\right\rangle
\left(i_g^{-1}\right)^{\sharp}\left(\alpha_{g}\left(\mu_g\left(m^{(-1)}\right)_1 \right)\right)c\otimes m^{(0)}\\
& = &  
\left\langle v,q_{G,H}(\mu_g\left(m^{(-1)}\right)_2)\right\rangle
\left(i_g^{-1}\right)^{\sharp}\left(\alpha_{g}\left(\mu_g\left(m^{(-1)}\right)_1 \right)\right)c\otimes m^{(0)}\\
& = &  
\left\langle v,q_{G_g,H}(q_{G,G_g}(\mu_g\left(m^{(-1)}\right)_2))\right\rangle
\left(i_g^{-1}\right)^{\sharp}\left(\alpha_{g}\left(\mu_g\left(m^{(-1)}\right)_1 \right)\right)c\otimes m^{(0)}\\
& = &  
\left\langle v,q_{G_g,H}(m^{(-1)}_2)\right\rangle
\left(i_g^{-1}\right)^{\sharp}\left(\alpha_{g}(\mu_g(m^{(-1)}_1))\right)c\otimes
m^{(0)}= \left\langle v,q_{G_g,H}(m^{(-1)})\right\rangle c\otimes m^{(0)}\\
& = & \left\langle \iota_{H,G_g}(v),m^{(-1)}\right\rangle
c\otimes m^{(0)}=c\otimes m\cdot \iota_{H,G_g}(v),
\end{eqnarray*}
as claimed, where we used (\ref{mu-gcolinear}) in the fourth to last equation and (\ref{mu-greta-g}) in the third to last equation.
\end{proof}

From the $G$-equivariant Hopf algebra map $B:\mathbf{k}[H]\to \mathscr{O}(K)$,  
we obtain the map
\begin{equation}\label{charBg}
B_g:\mathbf{k}[H]\to \mathbf{k},\quad v\mapsto \langle B(v),g\rangle,
\end{equation}
which is a $G_g$-invariant character since for every $u\in \mathbf{k}[G_g]$ and $v\in \mathbf{k}[H]$,
$$B_g({\rm ad_{\ell}}(S(u))(v))=\langle B({\rm ad_{\ell}}(S(u))(v)),g\rangle=\langle B(v),{\rm ad_{\ell}}(u)(g)\rangle=\langle B(v),g\rangle=B_g(v).$$
We also obtain the Drinfeld twist $\psi_g$ for $\mathscr{O}(G_g/H)$ (i.e., a Hopf $2$-cocycle for $\mathbf{k}[G_g/H]$), given by
\begin{equation}\label{twistJg}
\psi_g(x,y)=\langle \sigma(x,y),g\rangle ;\quad\forall x,y\in \mathbf{k}[G_g/H].
\end{equation}
Let $\mathbf{k}^{\psi_g}[G_g/H]=(\mathscr{O}(G_g/H)_{\psi_g})^*$ be the associated twisted group algebra, and let $\cdot$ denote its multiplication. 

\begin{lemma}\label{Help0}
For each $g\in K(\mathbf{k})$, the map 
$$p_g:\mathbf{k}[G_g]\to \mathbf{k}^{\psi_g}[G_g/H],\quad u\mapsto B_g(\eta_H(u_1))\pi_H(u_2),$$
is a surjective algebra homomorphism, which maps every $v\in \mathbf{k}[H]$ to $B_g(v)1$.
\end{lemma}

\begin{proof}
By Lemma \ref{gcphgh}(3) and Lemma \ref{newlemma}(1), for every $u,\tilde{u}\in \mathbf{k}[G_g]$, we have
\begin{eqnarray*}
\lefteqn{p_g(u\tilde{u})=B_g(\eta(u_1\tilde{u}_1))\pi(u_2\tilde{u}_2)}\\
& = & B_g\left(\eta(u_1)\gamma(\pi(u_2))_1\eta(\tilde{u}_1)S(\gamma(\pi(u_2))_2)\eta\{\gamma(\pi(u_2))_3\gamma(\pi(\tilde{u}_2))\}\right)
\pi(u_3\tilde{u}_3)\\
& = & B_g\left(\eta(u_1)\eta(\tilde{u}_1)\eta\{\gamma(\pi(u_2))\gamma(\pi(\tilde{u}_2))\}\right)
\pi(u_3\tilde{u}_3)\\
& = & B_g(\eta(u_1)\eta(\tilde{u}_1))\langle \sigma(\pi(u_2),\pi(\tilde{u}_2)),g\rangle \pi(u_3\tilde{u}_3)\\
& = & B_g(\eta(u_1)\eta(\tilde{u}_1))\psi_g(\pi(u_2),\pi(\tilde{u}_2))\pi(u_3\tilde{u}_3)\\
& = & B_g(\eta(u_1)\eta(\tilde{u}_1))\pi(u_2)\cdot \pi(\tilde{u}_2)= p_g(u)\cdot p_g(\tilde{u}), 
\end{eqnarray*}
as claimed.
\end{proof}

The injective coalgebra homomorphism $p_g^*:\mathscr{O}(G_g/H)_{\psi_g}\xrightarrow{1:1}\mathscr{O}(G_g)$ induces an Abelian embedding 
${\rm Corep}(\mathscr{O}(G_g/H)_{\psi_g})\subseteq {\rm Corep}(\mathscr{O}(G_g))$, so that 
${\rm Corep}(\mathscr{O}(G_g/H)_{\psi_g})$ consists of all $G_g$-modules $M$ for which
\begin{equation}\label{hinv}
m\cdot \iota_{H,G_g}(v)=B_g(v)m;\quad \forall v\in \mathbf{k}[H],\,m\in M.
\end{equation}

Now for each $C_g\in {\rm C}_K(\mathbf{k})$, consider the full Abelian subcategory 
$$
\widetilde{\mathscr{Z}}(K,H,B)_{C_g}:=\mathbf{F}_{C_g}({\rm Corep}(\mathscr{O}(G_g/H)_{\psi_g}))\subseteq \mathscr{Z}(G)_{C_g},
$$
and let
\begin{equation}\label{Dg}
\widetilde{\mathscr{Z}}(K,H,B):=\bigoplus_{C_g\in {\rm C}_K(\mathbf{k})}\overline{\widetilde{\mathscr{Z}}(K,H,B)_{C_g}}\subseteq \mathscr{Z}(G).
\end{equation} 

\begin{lemma}\label{Help}
Fix $\widetilde{\mathscr{Z}}:=\widetilde{\mathscr{Z}}(K,H,B)$ as in (\ref{Dg}). Then the following hold:
\begin{enumerate}
\item
For each simple $M\in {\rm Corep}(\mathscr{O}(G_g/H)_{\psi_g})$, we have
$$P_{\widetilde{\mathscr{Z}}}(\mathbf{F}_{C_g}(M))=\mathscr{O}(K^{\circ})\ot\mathscr{O}(C_g(\mathbf{k}))\ot P_{(G_g/H,\psi_g)}(M).$$
\item
The Frobenius-Perron dimension of the Abelian subcategory $\widetilde{\mathscr{Z}}\subseteq \mathscr{Z}(G)$ is equal to $|K|[G:H]$.
\item
For each $C_g\in {\rm C}_K(\mathbf{k})$, we have $\overline{\widetilde{\mathscr{Z}}_{C_g}}=\mathscr{Z}(K,H,B)\cap \overline{\mathscr{Z}(G)_{C_g}}$. Thus,  
$$\mathscr{Z}(K,H,B)=\bigoplus_{C_g\in {\rm C}_K(\mathbf{k})}\mathscr{Z}(K,H,B)\cap \overline{\mathscr{Z}(G)_{C_g}}.$$
\end{enumerate}
\end{lemma} 

\begin{proof}
(1) The proof is similar to the proof of \cite[Theorem 8.3(5)]{GS}.

(2) By definition, using (1), we have
\begin{eqnarray*}
\lefteqn{{\rm FPdim}(\widetilde{\mathscr{Z}})=\sum_{X\in \mathcal{O}(\widetilde{\mathscr{Z}})} {\rm FPdim}(X){\rm FPdim}(P_{\widetilde{\mathscr{Z}}}(X))}\\
& = & \sum_{C_g\in {\rm C}_K(\mathbf{k}),\,M\in \mathcal{O}(G_g/H,\psi_g)} {\rm FPdim}(\mathbf{F}_{C_g}(M)){\rm FPdim}(P_{\widetilde{\mathscr{Z}}}(\mathbf{F}_{C_g}(M)))\\
& = & \sum_{C_g\in {\rm C}_K(\mathbf{k}),\,M\in \mathcal{O}(G_g/H,\psi_g)} |C_g|{\rm dim}_{\mathbf{k}}(M)|K^{\circ}||C_g(\mathbf{k})|{\rm dim}_{\mathbf{k}}\left(P_{(G_g/H,\psi_g)}(M)\right)\\
& = & \sum_{C_g\in {\rm C}_K(\mathbf{k})} |C_g||K^{\circ}||C_g(\mathbf{k})|[G_g:H]=\sum_{C_g\in {\rm C}_K(\mathbf{k})} 
[G:H]|K^{\circ}||C_g(\mathbf{k})|\\
& = & [G:H]|K|\sum_{C_g\in {\rm C}_K(\mathbf{k})} 
\frac{|C_g(\mathbf{k})|}{|K(\mathbf{k})|}=[G:H]|K|,
\end{eqnarray*}
as claimed.

(3) If $\mathbf{F}_{C_g}(M,\rho_M)\in \mathscr{Z}(K,H,B)\cap \mathscr{Z}(G)_{C_g}$, then by Lemma \ref{expac},
$$\left\langle \iota_{H,G_g}(v),m^{(-1)}\right\rangle
c\otimes m^{(0)}=q_{K,C_g}(B(v))c\ot m$$
for every $v\in\mathbf{k}[H]$ and $c\otimes m\in \mathscr{O}(C_g)\otimes M$. Applying $g\ot\id_M$ to both sides yields 
$$\left\langle \iota_{H,G_g}(v),m^{(-1)}\right\rangle 
c(g)m^{(0)}=B_g(v)c(g)m,$$
which implies that $(M,\rho_M)\in {\rm Corep}(\mathscr{O}(G_g/H)_{\psi_g})$, so 
$$\mathscr{Z}(K,H,B)\cap \mathscr{Z}(G)_{C_g}\subseteq \widetilde{\mathscr{Z}}(K,H,B)_{C_g}.$$
Thus, $\mathscr{Z}(K,H,B)\cap \overline{\mathscr{Z}(G)_{C_g}}\subseteq \overline{\widetilde{\mathscr{Z}}(K,H,B)_{C_g}}$, so $\mathscr{Z}(K,H,B)\subseteq \widetilde{\mathscr{Z}}(K,H,B)$. Since by (2), $\widetilde{\mathscr{Z}}(K,H,B)$ and $\mathscr{Z}(K,H,B)$ have the same Frobenius-Perron dimension, the claim follows.  
\end{proof}
 
\begin{theorem}\label{main1}
Fix a braided subcategory $\mathscr{Z}:=\mathscr{Z}(K,H,B)\subseteq \mathscr{Z}(G)$, and for each $C_g\in {\rm C}_K(\mathbf{k})$, set $\mathscr{Z}_{C_g}:=\mathscr{Z} \cap \mathscr{Z}(G)_{C_g}$. Then the following hold:
\begin{enumerate}
\item
The Abelian equivalence $\mathbf{F}_{C_g}$ (\ref{funfc}) restricts to an equivalence 
$$\mathbf{F}_{C_g}:{\rm Corep}(\mathscr{O}(G_g/H)_{\psi_g})\xrightarrow{\simeq}\mathscr{Z}_{C_g},\quad  
(M,\rho_M)\mapsto \left(\mathscr{O}(C_g)\otimes M,\rho_M^g\right).$$
In particular, the composition functor 
$$\Rep(G/H)\simeq {\rm Corep}(\mathscr{O}(G/H))\xrightarrow{\mathbf{F}_{1}}\mathscr{Z}_1\hookrightarrow \mathscr{Z}$$
coincides with the canonical embedding $\Rep(G/H)\hookrightarrow \mathscr{Z}$ of braided categories.
\item
For each $C_g\in {\rm C}_K(\mathbf{k})$, with representative $g\in C_g(\mathbf{k})$, there is a bijection between the set of equivalence classes of simple objects $M\in {\rm Corep}(\mathscr{O}(G_g/H)_{\psi_g})$ and isomorphism classes of simple objects of $\mathscr{Z}_{C_g}$, assigning $M$ to $\mathbf{F}_{C_g}(M)$. Moreover, we have a direct sum decomposition of Abelian categories
$$\mathscr{Z}(K,H,B)=\bigoplus_{C_g\in {\rm C}_K(\mathbf{k})}\overline{\mathscr{Z}_{C_g}},$$
and
$\overline{\Rep(G/H)}\simeq \overline{\mathscr{Z}_{1}}=\mathscr{Z}\left(K^{\circ},H,\iota^{\sharp}_{K^{\circ},K}\circ B\right)$. In particular, if $G$ is connected then the simples of $\mathscr{Z}$ are precisely those of $\Rep(G/H)\simeq {\rm Corep}(\mathscr{O}(G/H))$.
\item
For each $(M,\rho_M)\in {\rm Corep}(\mathscr{O}(G_g/H)_{\psi_g})$, we have $\mathbf{F}_{C_g}(M,\rho_M)^*\cong \mathbf{F}_{C_{g^{-1}}}(M^*,\rho_{M^*})$.
\item
For each $(M,\rho_M)\in {\rm Corep}(\mathscr{O}(G_g/H)_{\psi_g})$, we have
$$
{\rm FPdim}\left(\mathbf{F}_{C_g}(M,\rho_M)\right)=|C_g|{\rm dim}_{\mathbf{k}}(M).$$
\item
For each simple $(M,\rho_M)\in {\rm Corep}(\mathscr{O}(G_g/H)_{\psi_g})$, we have
$$P_{\mathscr{Z}}\left(\mathbf{F}_{C_g}(M,\rho_M)\right)\cong 
\left(\mathscr{O}(K^{\circ})\ot\mathscr{O}(C_g(\mathbf{k}))\ot P_{(G_g/H,\psi_g)}(M,\rho_M),R_M^g\right),$$
where $\mathscr{O}(K)$ acts on the first two factors diagonally, and  
$$R_{M}^g:=\left(\chi_{C_g}\ot\id^{\ot 2}\right)\left(\id_{\mathscr{O}(K^{\circ})}\ot\rho_{P_{(G_g/H,\psi_g)}(M,\rho_M)}^g\right)\left(\nu_{C_g}\ot\id\right).$$ 
In particular, we have 
$${\rm FPdim}\left(P_{\mathscr{Z}}\left(\mathbf{F}_{C_g}(M,\rho_M)\right)\right)=|K^{\circ}||C_g(\mathbf{k})|{\rm dim}_{\mathbf{k}}\left(P_{(G_g/H,\psi_g)}(M,\rho_M)\right).$$
\item
For each $C_g\in {\rm C}_K(\mathbf{k})$, we have 
$
{\rm FPdim}\left(\overline{\mathscr{Z}_{C_g}}\right)=|K^{\circ}||C_g(\mathbf{k})|[G:H]$. 
\end{enumerate}
\end{theorem}

\begin{proof}
(1)-(2) These follow from Lemma \ref{Help}.

(3) This follows from \cite[Theorem 8.3]{GS} (see Theorem \ref{maings}).

(4) This follows from (1).

(5)-(6) These follow from \cite[Theorem 8.3]{GS} and the previous parts.
\end{proof}

\section{Some special cases and examples}\label{sec:examples}

\subsection{Canonical examples}\label{sec:canexs}
The following are some canonical Hopf quotients of $D(G)$ and their corresponding braided subcategories of $\mathscr{Z}(G)$:

(1) We have 
$$D(1,G,1)=\mathbf{k},\quad D(G,1,1)=D(G),\quad D(1,1,1)=\mathbf{k}[G].$$
Thus, we have
$$\mathscr{Z}(1,G,1)={\rm Vect},\quad \mathscr{Z}(1,G,1)'=\mathscr{Z}(G,1,1)=\mathscr{Z}(G),\quad \mathscr{Z}(1,1,1)=\Rep(G).$$

(2) For every normal subgroup scheme $K\subseteq G$, we have $D(1,K,1)=\mathbf{k}[G/K]$, and $D(K,1,1)=\mathscr{O}(K)^{{\rm cop}}\# \mathbf{k}[G]$ is a tensor product coalgebra and smash product algebra. In particular, $D(1,G^{\circ},1)=\mathbf{k}[G(\mathbf{k})]$, and $D(G^{\circ},1,1)=\mathscr{O}(G^{\circ})^{{\rm cop}}\# \mathbf{k}[G]$.

(3) For every normal subgroup scheme $K\subseteq G$, a {\em central} subgroup scheme $H\subseteq G$, and a $G$-equivariant group scheme morphism $B:K\to H^{\vee}$, the Hopf algebra
$$D(K,H,B)=\mathscr{O}(K)\#_{\sigma}^{\tau} \mathbf{k}[G/H]$$
is a tensor product coalgebra and smash product algebra, and by Proposition \ref{genproprs}(7), 
$$D(H,K,\overline{B})=\mathbf{k}[H^{\vee}]\#_{\sigma}^{\tau} 
\mathbf{k}[G/K]=\mathbf{k}[\widetilde{G}],$$ 
where $\widetilde{G}$ is the group scheme extension of $H^{\vee}$ by $G/K$ corresponding to  
$$\sigma\in Z^2(H^{\vee},G/K),\quad \tau\in Z^1(H^{\vee},G/K),$$
with trivial action. Thus, we have 
$$\mathscr{Z}(K,H,B)=\Rep\left(\mathscr{O}(K)\#_{\sigma}^{\tau} 
\mathbf{k}[G/H]\right),$$
and $\mathscr{Z}(K,H,B)'=\mathscr{Z}(H,K,\overline{B})=\Rep(\widetilde{G})$.

\subsection{Direct products}\label{sec:canexs5}
Let $G:=K\times H$, and let $B:\mathbf{k}[H]\xrightarrow{}\mathscr{O}(K)$ be any $G$-equivariant Hopf algebra map. Then
$$\mathscr{Z}(K\times 1,1\times H,B)\subseteq \mathscr{Z}(G)$$
is a non-degenerate braided subcategory. Note that since $\sigma$ and $\tau$ are trivial by Remark \ref{trivial}, it follows that  
$$\mathscr{Z}(K\times 1,1\times H,B)\cong \mathscr{Z}(K)$$
as braided categories (regardless of the choice of $B$).

\subsection{Constant groups}\label{sec:exconst1}
Assume that $G$ is constant; that is, $G=G(\mathbf{k})$. For $g\in G$, let $\delta_g\in \mathscr{O}(G)$ be the delta function at $g$. The elements $\{\delta_g\}_{g\in G}$ are pairwise orthogonal idempotents, hence any $\mathscr{O}(G)$-module $V$ decomposes as $$V = \bigoplus_{g\in G} V_g, \qquad V_g:=\delta_gV.$$ Recall that in this case a $D(G)$-module is a $G$-graded vector space $V=\bigoplus_{g\in G}V_g$, together with an action of $G$, such that $x\cdot V_g \subseteq V_{xgx^{-1}}$ for every $x,g\in G$.
	
Now fix a triple $(K,H,B)$ and let $\theta:D(G)\twoheadrightarrow D(K,H,B)$ be the quotient map as in Theorem~\ref{quotientha}. Since in this case $\gamma_H$ is a coalgebra map, we have that
$$D(K,H,B)=\mathscr{O}(K)^{{\rm cop}}\#_{\sigma}\mathbf{k}[G/H]$$
is a tensor product coalgebra by Proposition \ref{genproprs}(4).

If $V\in\Rep(D(K,H,B))=\mathscr{Z}(K,H,B)$, and we inflate $V$ along $\theta$ to a $D(G)$-module, then the $\mathscr{O}(G)$-action factors through $q_K:\mathscr{O}(G)\twoheadrightarrow\mathscr{O}(K)$. In particular, $q_K(\delta_g)=0$ for $g\notin K$, so $V_g=\delta_gV=0$ unless $g\in K$. Thus, $V$ is supported on $K$. Moreover, for $h\in H$, we have $\pi_H(h)=1$ and $\eta_H(h)=h$, hence 
$$\theta(1\bowtie h)=B(h)\# 1\in \mathscr{O}(K)\#_{\sigma}\mathbf{k}[G/H],$$ 
so on each homogeneous component $V_k$, $k\in K$, the element $h$ acts by the scalar $B(h)(k)$. Equivalently, setting 
$$\beta:H\times K\to \mathbf{k}^\times,\qquad \beta(h,k):=B(h)(k),$$ 
we see that $\mathscr{Z}(K,H,B)$ consists precisely of $G$-equivariant sheaves on $G$ supported on $K$ for which the action of $H$ on the fibre over $k\in K$ is given by $\beta(h,k)$, so in the fusion case, $\mathscr{Z}(K,H,B)$ is the category $\cS(K,H,B)$ from \cite[Section 3]{NNW}.

It follows from Theorem \ref {main1}, similarly to \cite[Corollary 6.4]{GS2}, that we have a direct sum decomposition of Abelian categories
$$\mathscr{Z}(K,H,B)=\bigoplus_{C_g\in {\rm C}_K}\mathscr{Z}(K,H,B)_{C_g}\simeq \bigoplus_{C_g\in {\rm C}_K}{\rm Corep}(\mathscr{O}(G_g/H)_{\psi_g}),$$
given by $\mathbf{F}_{C_g}$, $C_g\in {\rm C}_K$,
and for every simple $M\in {\rm Corep}(\mathscr{O}(G_g/H)_{\psi_g})$, we have
$$P_{\mathscr{Z}(K,H,B)}\left(\mathbf{F}_{C_g}(M)\right)= \mathbf{F}_{C_g}(P_{(G_g/H,\psi_g)}(M))=
\left(\mathscr{O}(C_g)\ot P_{(G_g/H,\psi_g)}(M),\rho_{P_{(G_g/H,\psi_g)}(M)}^g\right),$$
where $\mathscr{O}(K)$ acts on the first factor.

\subsection{Connected groups}\label{sec:restlie case double}
Assume that $G$ is {\em connected}; that is, $G(\mathbf{k})=1$ (e.g., $G$ is the finite group scheme associated to 
a finite dimensional restricted $p$-Lie algebra $\mathfrak{g}$). Then it follows from Theorem \ref {main1}, similarly to \cite[Example 7.5]{GS2}, that for each triple $\left(K,H,B\right)$ as in Theorem \ref{quotientha}, we have $\mathscr{Z}(K,H,B)=\overline{{\rm Rep}\left(G/H\right)}$, and 
for every simple $M\in \Rep\left(G/H\right)$, we have
$$P_{\mathscr{Z}(K,H,B)}\left(\mathbf{F}(M)\right)\cong 
\left(\mathscr{O}(K)\ot P_{G/H}(M),\id\ot\rho_{P_{G/H}(M)}\right),$$
where $\mathscr{O}(K)$ acts on the first factor.

\begin{example}\label{2dimrestlie}
Let $\mathfrak{g}$ be the unique $2$-dimensional non-Abelian restricted $p$-Lie algebra with basis $\{x,y\}$, such that $[x,y]=y$, $x^{[p]}=x$, and $y^{[p]}=0$. Then $\mathfrak{g}={\rm Lie}(G_1)$, where $G_1$ is the Frobenius kernel of the group scheme $G:=\mathbb{G}_a\rtimes \mathbb{G}_m$ of automorphisms of the affine line $\mathbb{A}^1$. It is clear that $\mathfrak{a}:={\rm sp}_{\mathbf{k}}\{y\}={\rm Lie}(\mathbb{G}_{a,1})\subseteq \mathfrak{g}$ is an ideal, $\mathfrak{g}/\mathfrak{a}={\rm Lie}(\mathbb{G}_{m,1})$, and $G_1=\mathbb{G}_{a,1}\rtimes \mathbb{G}_{m,1}$.

We have the (degenerate non-symmetric) braided subcategory 
$$\mathscr{Z}\left(\mathbb{G}_{a,1},1,1\right)=\Rep(\mathbb{G}_{a,1}^{\vee}\times G_1)\subseteq \mathscr{Z}(G_1)$$
of Frobenius-Perron dimension $p^3$, 
and the infinite family of non-equivalent braided subcategories
$$\mathscr{Z}_{\lambda}:=\mathscr{Z}\left(\mathbb{G}_{a,1},\mathbb{G}_{a,1},B_{\lambda}\right)=\Rep(\mathbb{G}_{a,1}^{\vee}\times \mathbb{G}_{m,1},R_{\lambda})\subseteq \mathscr{Z}(G_1)$$ 
of Frobenius-Perron dimension $p^2$ parameterized by $\lambda \in \mathbf{k}$, where $B_{\lambda}$ and $R_{\lambda}$ are given in \S\ref{sec:appendix1}. Note that $\mathscr{Z}_{\lambda}$ is Lagrangian if and only if $\lambda=0$, and that $\mathscr{Z}_0$ and $\mathscr{Z}\left(1,1,1\right)=\Rep(G_1,1\ot 1)$ are not equivalent as braided subcategories of $\mathscr{Z}(G_1)$, as predicted by Theorem \ref{mainclass}. \qed 
\end{example}

\begin{example}\label{3dimrestHlie}
Let $\mathfrak{g}$ be the $3$-dimensional Heisenberg $p$-Lie algebra with basis $\{x,y,z\}$, such that $[x,y]=z$. Let 
$\mathfrak{h}:={\rm sp}_{\mathbf{k}}\{z\}\subset \mathfrak{g}$, and let $\mathfrak{k}:={\rm sp}_{\mathbf{k}}\{x,z\}\subset \mathfrak{g}$; the former is a central ideal, and the latter is an ideal. Recall that for $p>2$, $\mathfrak{g}$ admits exactly three non-isomorphic $[p]$-structures, given by 
$$x^{[p]}=y^{[p]}=z^{[p]}=0;\quad x^{[p]}=z,\,y^{[p]}=z^{[p]}=0;\quad
x^{[p]}=y^{[p]}=0,\,z^{[p]}=z.$$

In each case, let $G$ be the finite group scheme such that $\mathscr{O}(G)=(u^{[p]}(\mathfrak{g}))^*$; it is a local commutative Hopf algebra of dimension $p^3$ (so, $u^{[p]}(\mathfrak{g})$ is connected). Also, let $H\subset G$ be the central subgroup scheme such that $\mathfrak{h}={\rm Lie}(H)$, and $K\subset G$ the normal subgroup scheme such that $\mathfrak{k}={\rm Lie}(K)$.

{\bf The case $x^{[p]}=y^{[p]}=z^{[p]}=0$:}   
In this case, $\mathfrak{g}$ is $[p]$-unipotent, so that $u^{[p]}(\mathfrak{g})$ is local (and connected). Thus, $D(G)=(u^{[p]}(\mathfrak{g}))^{*{\rm cop}}\bowtie u^{[p]}(\mathfrak{g})$ is local and connected. Also, $H\cong \mathbb{G}_{a,1}$, $G/H\cong \mathbb{G}_{a,1}^2$, $K\cong \mathbb{G}_{a,1}^2$, and $G/K\cong \mathbb{G}_{a,1}$.

(1) We have an infinite family of non-equivalent {\em unipotent} braided subcategories 
$$\mathscr{Z}\left(H,H,B_{\lambda}\right)=\Rep(\mathbf{k}[\mathbb{G}_{a,1}^{\vee}]\#_{\sigma}^{\tau}\mathbf{k}[\mathbb{G}_{a,1}^2],R_{\lambda})=\Rep(\widetilde{G},R_{\lambda})\subseteq \mathscr{Z}(G)$$
of Frobenius-Perron dimension $p^3$, parameterized by $\lambda \in \mathbf{k}$, where $\widetilde{G}$ is the finite group scheme extension of $\mathbb{G}_{a,1}^2$ by $\mathbb{G}_{a,1}^{\vee}$ corresponding to $\tau\in Z^1(\mathbb{G}_{a,1}^2,(\mathbb{G}_{a,1}^{\vee})^2)$ and $\sigma\in Z^2(\mathbb{G}_{a,1}^2,\mathbb{G}_{a,1}^{\vee})$, and $B_{\lambda}$ and $R_{\lambda}$ are given in \S\ref{sec:appendix1}.

(2) We have an infinite family of non-equivalent {\em unipotent} braided subcategories 
$$\mathscr{Z}(K,H,B)=\Rep(\mathbf{k}[(\mathbb{G}_{a,1}^{\vee})^2]\#_{\sigma}^{\tau}\mathbf{k}[\mathbb{G}_{a,1}^2],R_B)=\Rep(\widetilde{G},R_B)\subseteq \mathscr{Z}(G)$$
of Frobenius-Perron dimension $p^4$, parameterized by $B\in \Hom(\mathbb{G}_{a,1},(\mathbb{G}_{a,1}^{\vee})^2)$, where $\widetilde{G}$ is the group scheme extension of $\mathbb{G}_{a,1}^2$ by $(\mathbb{G}_{a,1}^{\vee})^2$ corresponding to $\sigma\in Z^2(\mathbb{G}_{a,1}^2,(\mathbb{G}_{a,1}^{\vee})^2)$ and $\tau\in Z^1(\mathbb{G}_{a,1}^2,(\mathbb{G}_{a,1}^{\vee})^4)$.

{\bf The case $x^{[p]}=z,\,y^{[p]}=z^{[p]}=0$:} 
In this case $H\cong \mathbb{G}_{a,1}$ and $G/H\cong \mathbb{G}_{m,1}\times \mathbb{G}_{a,1}$, so we have the infinite family of non-equivalent braided subcategories 
$$\mathscr{Z}\left(H,H,B_{\lambda}\right)=\Rep(\mathbf{k}[\mathbb{G}_{a,1}^{\vee}]\#_{\sigma}^{\tau}\mathbf{k}[\mathbb{G}_{m,1}\times \mathbb{G}_{a,1}],R_{\lambda})=\Rep(\widetilde{G},R_{\lambda})\subseteq \mathscr{Z}(G)$$
of Frobenius-Perron dimension $p^3$, parameterized by $\lambda \in \mathbf{k}$, where $\widetilde{G}$ is the group scheme extension of $\mathbb{G}_{m,1}\times \mathbb{G}_{a,1}$ by $\mathbb{G}_{a,1}^{\vee}$ corresponding to $\sigma\in Z^2(\mathbb{G}_{m,1}\times \mathbb{G}_{a,1},\mathbb{G}_{a,1}^{\vee})$ and $\tau\in Z^1(\mathbb{G}_{m,1}\times \mathbb{G}_{a,1},(\mathbb{G}_{a,1}^{\vee})^2)$, and $B_{\lambda}$ and $R_{\lambda}$ are given in \S\ref{sec:appendix1}.

{\bf The case $x^{[p]}=y^{[p]}=0,\,z^{[p]}=z$:}  
In this case $H\cong \mathbb{G}_{m,1}$ and $G/H\cong \mathbb{G}_{a,1}^2$. Since the only group scheme morphism $\mathbb{G}_{a,1}\to \mathbb{Z}/p\mathbb{Z}$ is the trivial one, we only have the symmetric subcategory 
$$\mathscr{Z}\left(H,H,1\right)=\Rep(\mathbb{Z}/p\mathbb{Z}\times \mathbb{G}_{a,1}^2,1\ot 1)\subseteq \mathscr{Z}(G)$$
of Frobenius-Perron dimension $p^3$. \qed 
\end{example}

\subsection{Trivial $B$}\label{sec:Beq1}
Set $\mathscr{Z}(K,H):=\mathscr{Z}(K,H,1)$. Since by definition, an object $X$ in ${\rm Coh}(K)^{G/H}$ is an $\mathscr{O}(K)$-module in the category $\Rep(G/H)$, it follows from Proposition \ref{genproprs}(5) that we have a natural embedding of tensor categories 
$${\rm Coh}(K)^{G/H}\xrightarrow{1:1} \mathscr{Z}(K,H).$$
Since both categories have the same Frobenius-Perron dimension, it follows that we have an equivalence of tensor categories
$$\mathscr{Z}(K,H)\simeq {\rm Coh}(K)^{G/H}.$$

Now consider $K$ as a subgroup scheme of $K\times G/H$ via the group scheme embedding 
$$\Delta_{K,H}:K\to K\times G/H,\quad k\mapsto (k,\pi_H(k)).$$
Then, similarly to the proof that $\mathscr{Z}(G)\simeq {\rm Coh}(G)^{G}\simeq \mathscr{C}(G\times G,\Delta_{G,1}(G))$ \cite[Section 8]{GS}, we have an equivalence of tensor categories 
\begin{equation*}
{\rm Coh}(K)^{G/H}\simeq \mathscr{C}(K\times G/H,\Delta_{K,H}(K)),
\end{equation*}
so $\mathscr{Z}(K,H)$ is group scheme-theoretical in the sense of \cite{Ge}. In particular, by \cite{N}, it follows that 
$\mathscr{Z}(K,H)$ has the finite generation property, and \cite{Ge} implies the classification of exact indecomposable module categories over $\mathscr{Z}(K,H)$.

\subsection{Commutative groups}\label{sec:canexs0}
Assume that $A\subseteq G$ is a normal {\em commutative} subgroup scheme (e.g. $A$ is contained in the center of $G$). Then $D(A,A,1)=\mathbf{k}[A^{\vee}\times G/A]$ is a triangular Hopf algebra, where $A^{\vee}$ is the Cartier dual of $A$, and 
$$\mathscr{Z}(A,A,1)=\Rep\left(A^{\vee}\times G/A\right)\subseteq \mathscr{Z}(G)$$
is a Lagrangian subcategory.

More generally, if $B:A\to A^{\vee}$ is any $G$-equivariant group scheme morphism, then $(D(A,A,B),R_B)=(\mathbf{k}[A^{\vee}]\#_{\sigma}^{\tau} \mathbf{k}[G/A],R_B)=(\mathbf{k}[\widetilde{G}],R_B)$ is a quasitriangular Hopf algebra (see Proposition \ref{genproprs}(7)), and by Corollary \ref{main4}(3), the braided subcategory 
$$\mathscr{Z}(A,A,B)=\Rep\left(\mathbf{k}[A^{\vee}]\#_{\sigma}^{\tau} \mathbf{k}[G/A],R_B\right)=\Rep(\widetilde{G},R_B)\subseteq \mathscr{Z}(G)$$
is Lagrangian if and only if $B=\overline{B}:=B^{\vee}S$.

Assume further that $G$ is {\em commutative}. Then for any group scheme morphism $B:A\xrightarrow{}A^{\vee}$, we have the braided subcategory 
$$\mathscr{Z}(A,A,B)=\Rep\left(\mathbf{k}[A^{\vee}]\#_{\sigma}^{\tau} \mathbf{k}[G/A],R_B\right)=\Rep(\widetilde{G},R_B)\subseteq \mathscr{Z}(G);$$
by Corollary \ref{main4}(3), it is Lagrangian if and only if $B=\overline{B}$.

Assume even further that $G$ is {\em self dual}. Then for each group scheme morphism $B:G\xrightarrow{}G^{\vee}$, by Corollary \ref{main4}(2), the braided subcategory 
$$\mathscr{Z}(G,G,B)=\Rep\left(\mathbf{k}[G^{\vee}],R_B\right)\subseteq \mathscr{Z}(G)$$
is non-degenerate if and only if $B\star B^{\vee}:G\xrightarrow{\cong }G^{\vee}$ is a group scheme isomorphism. 

For example, let $G:=\mathbb{G}_{a,1}$. In \S\ref{sec:appendix1} we show that 
$$\mathscr{Z}(\mathbb{G}_{a,1},\mathbb{G}_{a,1},B_{\lambda})=\Rep\left(\mathbf{k}[\mathbb{G}_{a,1}^{\vee}],R_{\lambda}\right)\subseteq \mathscr{Z}(\mathbb{G}_{a,1})$$
is an infinite family of non-equivalent non-degenerate braided subcategories of Frobenius-Perron dimension $p$, parametrized by $\lambda\in \mathbf{k}^{\times}$, where $R_{\lambda}$ is given in (\ref{rz1}).

\begin{example}\label{newexs}
As mentioned in Remark \ref{trivialtaubar}, if $\eta$ or $\gamma$ is a coalgebra map, then $\tau$ is trivial. In \S\ref{sec:appendix2} we give an example of a nontrivial $\tau$ for $G= \mathbb{G}_{a,2}$ and $H = K = \mathbb{G}_{a,1}$ (\ref{tauappendix}), and deduce that for each $\lambda\in \mathbf{k}^{\times}$, the quasitriangular Hopf algebra 
$$\left(D(\mathbb{G}_{a,1},\mathbb{G}_{a,1},B_{\lambda}),R_{\lambda}\right)= \left(\mathbf{k}[\mathbb{G}_{a,1}^{\vee}]\#^{\tau} \mathbf{k}[\mathbb{G}_{a,2}/\mathbb{G}_{a,1}],R_{\lambda}\right)=\left(\mathbf{k}[\widetilde{\mathbb{G}_{a,1}}],R_{\lambda}\right)$$
is neither a tensor product coalgebra nor a tensor product algebra, where $R_{\lambda}$ is given in (\ref{rz2}). Thus,   
$$\mathscr{Z}(\mathbb{G}_{a,1},\mathbb{G}_{a,1},B_{\lambda})=\Rep\left(\mathbf{k}[\mathbb{G}_{a,1}^{\vee}]\#^{\tau} \mathbf{k}[\mathbb{G}_{a,2}/\mathbb{G}_{a,1}],R_{\lambda}\right)=\Rep\left(\widetilde{\mathbb{G}_{a,1}},R_{\lambda}\right)\subseteq \mathscr{Z}(\mathbb{G}_{a,2})$$
is an infinite family of non-equivalent braided subcategories of Frobenius-Perron dimension $p^2$, parametrized by $\lambda\in \mathbf{k}$. 

Note that since $\sigma$ is trivial, we have $\widetilde{\mathbb{G}_{a,1}}=\mathbb{G}_{a,1}^2$. \qed
\end{example}

\section{Appendix}\label{sec:appendix}

\subsection{The group scheme $\mathbb{G}_{a,r}$} 
Let $\mathbb{G}_{a,r}$ denote the $r$-th Frobenius kernel of $\mathbb{G}_a$. Recall that $\mathscr{O}(\mathbb{G}_{a,r}) = \mathbf{k}[t]/(t^{p^r})$ with $t$ primitive. Fix the basis $\{t^i\}_{0\le i\le p^r-1}$ for $\mathscr{O}(\mathbb{G}_{a,r})$ and denote by $\{\delta_i\}_{0\le i\le p^r-1}$ the dual basis. The multiplication in $\mathbf{k}[\mathbb{G}_{a,r}]$ is determined by 
$$\delta_m\delta_n=\begin{cases}
\binom{m+n}{m}\delta_{m+n},& m+n<p^r\\
0,& \text{otherwise}.
\end{cases}$$
The unit is $\delta_0$, and one has $$\Delta(\delta_n)=\sum_{a+b=n}\delta_a\otimes\delta_b,\quad \varepsilon(\delta_n)=\delta_{n,0},\quad S(\delta_n)=(-1)^n\delta_n.$$ 
It is well-known that the map
$$F_r: \mathbf{k}[X_0,\ldots,X_{r-1}]/(X_0^p,\ldots,X_{r-1}^p) \to \mathbf{k}[\mathbb{G}_{a,r}],\quad X_i\mapsto \delta_{p^i},$$ is an {\em algebra} isomorphism with inverse given by 
$$F_r^{-1}:\mathbf{k}[\mathbb{G}_{a,r}]\to \mathbf{k}[X_0,\ldots,X_{r-1}]/(X_0^p,\ldots,X_{r-1}^p),\quad \delta_n\mapsto \frac{1}{a_{0}! \cdots a_{r-1}!} X_0^{a_0} \cdots X_{r-1}^{a_{r-1}},$$ where $n = a_0 + \cdots + a_{r-1}p^{r-1}$ is the $p$-adic expansion of $n$.

\subsection{The center $\mathscr{Z}(\mathbb{G}_{a,1})$}\label{sec:appendix1}
Since $F_1^{-1}:\mathbf{k}[\mathbb{G}_{a,1}] \xrightarrow{\cong }\mathscr{O}(\mathbb{G}_{a,1})$ is a Hopf algebra isomorphism, it follows that Hopf algebra maps $\mathbf{k}[\mathbb{G}_{a,1}] \to \mathscr{O}(\mathbb{G}_{a,1})$ are parameterized by elements $\lambda \in \mathbf{k}$, where 
\begin{equation}\label{bz}
B_{\lambda}: \mathbf{k}[\mathbb{G}_{a,1}] \to \mathscr{O}(\mathbb{G}_{a,1}),\quad \delta_n\mapsto \frac{\lambda^n}{n!}t^n.
\end{equation}
This is an isomorphism if and only if $\lambda \neq 0$. Note that $B_{\lambda}=\overline{B_{\lambda}}$ if and only if $\lambda=0$ or $p=2$ (since $B_{\lambda}=B_{\lambda}^*$).

It now follows from Corollary \ref{main0} in a straightforward manner that for each $\lambda\in\mathbf{k}$, 
\begin{equation}\label{rz1}
R_{\lambda}:=\sum_{i=0}^{p-1}\frac{\lambda^i}{i!}(t^i\ot t^i)
\end{equation}
is an $R$-matrix for $D(\mathbb{G}_{a,1},\mathbb{G}_{a,1},B_{\lambda})=\mathbf{k}[\mathbb{G}_{a,1}^{\vee}]$. By Corollary \ref{main4}(2), $(\mathbf{k}[\mathbb{G}_{a,1}^{\vee}],R_{\lambda})$ is factorizable if and only if $\lambda \neq 0$; that is, $R_{\lambda}\ne 1\ot 1$. 

\subsection{The center $\mathscr{Z}(\mathbb{G}_{a,2})$}\label{sec:appendix2} 
By Theorem \ref{quotientha}, since $\mathbb{G}_{a,2}$ is commutative, Hopf quotient pairs of $D(\mathbb{G}_{a,2})$ are indexed by $(K,H,B)$, where $K$ and $H$ are subgroup schemes of $\mathbb{G}_{a,2}$, and $B : \mathbf{k}[H] \to \mathscr{O}(K)$ is a Hopf algebra map. 
We will now compute the Hopf quotient pairs for $K = H = \mathbb{G}_{a,1}$, up to equivalence, following the construction laid out in \S \ref{sec:Hopf algebra quotients of D(G)}. 

We have 
$$q:=q_{\mathbb{G}_{a,1}} : \mathscr{O}(\mathbb{G}_{a,2}) \twoheadrightarrow \mathscr{O}(\mathbb{G}_{a,1}),\quad t\mapsto  t,$$ with section 
$$\mu:=\mu_{\mathbb{G}_{a,1}} : \mathscr{O}(\mathbb{G}_{a,1}) \xrightarrow{1:1} \mathscr{O}(\mathbb{G}_{a,2}),\quad t\mapsto  t.$$ 

Next consider the exact sequence of Hopf algebras 
$$\mathbf{k} \to \mathbf{k}[\mathbb{G}_{a,1}] \to \mathbf{k}[\mathbb{G}_{a,2}] \xrightarrow{\pi:=\pi_{\mathbb{G}_{a,1}}} \mathbf{k}[\mathbb{G}_{a,2}/\mathbb{G}_{a,1}] \to \mathbf{k}.$$ 

\begin{lemma}
The following hold:
\begin{enumerate}
\item
The surjective Hopf algebra map $\pi$ is determined by 
$$\delta_n \mapsto \begin{cases}
\delta_{n/p},& p\mid n\\
0,& p\nmid n
\end{cases}$$
\item
The map $\gamma:=\gamma_{\mathbb{G}_{a,1}}$ determined by 
$$\gamma: \mathbf{k}[\mathbb{G}_{a,2}/\mathbb{G}_{a,1}] \to \mathbf{k}[\mathbb{G}_{a,2}],\quad \delta_n \mapsto  \delta_{pn},$$
is a cleaving map.
\item
$\gamma$ is {\em not} a coalgebra map.
\item
The convolution inverse of $\gamma$ is given by $\gamma^{-1}(\delta_n)=(-1)^n\delta_{pn}$.
\item
The action $\cdot$ of $\mathbf{k}[\mathbb{G}_{a,2}/\mathbb{G}_{a,1}]$ on $\mathscr{O}(\mathbb{G}_{a,1})$ is trivial. 
\item
$\eta:=\eta_{\mathbb{G}_{a,1}} = \id \star \gamma^{-1}\pi$ is the projection of $\mathbf{k}[\mathbb{G}_{a,2}]$ onto the subspace $\mathbf{k}[\mathbb{G}_{a,1}]$.
\item
The convolution inverse of $\eta$ is given by 
$$\eta^{-1}(\delta_n)=
\begin{cases}
(-1)^n\delta_n,& 0\le n<p,\\
0,& p\le n<p^2.
\end{cases}$$
\end{enumerate}
\end{lemma}

\begin{proof}
(1) It is easy to see that $\mathbf{k}[\mathbb{G}_{a,1}]^+ = \text{sp}_{\mathbf{k}}\{\delta_1,\ldots,\delta_{p-1}\}$, and moreover that $$\mathbf{k}[\mathbb{G}_{a,1}]^+ \mathbf{k}[\mathbb{G}_{a,2}] = \text{sp}_{\mathbf{k}}\{\delta_n \mid p \nmid n\}.$$ The $\mathbf{k}$-linear map $\mathbf{k}[\mathbb{G}_{a,2}] \to \mathbf{k}[\mathbb{G}_{a,1}]$ determined by $$\delta_n \mapsto \begin{cases}
\delta_{n/p},& p\mid n\\
0,& p\nmid n
\end{cases}$$ is a surjective algebra map with kernel $\mathbf{k}[\mathbb{G}_{a,1}]^+ \mathbf{k}[\mathbb{G}_{a,2}]$, so $\mathbf{k}[\mathbb{G}_{a,2}/\mathbb{G}_{a,1}]\cong \mathbf{k}[\mathbb{G}_{a,1}]$ and this map is precisely the Hopf algebra map $\pi$, as claimed.

(2) - (4) These are straightforward verification, and (5) is clear.

(6) For every $0\le n<p^2$, we have 
$$\eta(\delta_n) = \sum_{a+b = n} \delta_a \gamma^{-1}\pi(\delta_b) = \sum_{a+pb' = n} (-1)^{b'} \delta_a \delta_{pb'} = \sum_{a + pb' = n} (-1)^{b'} \binom{n}{a} \delta_{n}.$$ Write $n = ps + r$ with $0 \le s,r < p$. Then pairs $(a,b')$ satisfying $a+pb' = n$ are of the form $(r+ip, s-i)$ for $0 \le i \le s$. Hence $$\eta(\delta_n) = \sum_{i=0}^s (-1)^{s-i} \binom{r+ps}{r+pi} \delta_n = \sum_{i=0}^s (-1)^{s-i} \binom{s}{i} \delta_n,$$ the last equality being another application of Lucas' theorem. This sum equals $0$ if $s > 0$, and equals $\delta_n$ if $s = 0$. In other words, we have
$$
\eta(\delta_n) = 
\begin{cases}
\delta_n,& 0\le n<p,\\
0,& p\le n<p^2,
\end{cases}$$ 
so the claim follows.

(7) Similar to (6).
\end{proof}

Recall the definitions of $\bar{\sigma}, \bar{\tau}$ and $\sigma, \tau$ given in (\ref{defnolsigma}), (\ref{defnoltau}) and (\ref{defnsigma}), (\ref{defntau}).
	
\begin{lemma}
The cocycle $\sigma$ is trivial and the cococycle $\tau$ satisfies $\tau(\delta_0) = 1 \otimes 1$, 
$$\tau(\delta_1) = \sum_{j=1}^{p-1} \frac{\lambda^p}{j! (p-j)!} t^j \otimes t^{p-j},$$ 
\label{tauappendix} and $\tau(\delta_n) = 0$ for $2 \le n < p$.
\end{lemma}

\begin{proof}
By the definition of $\bar{\sigma}$, $\bar{\sigma}(\delta_m,\delta_n)$ is an integer multiple of $\delta_{p(m+n)}$. Since the image of $\bar{\sigma}$ is in $\mathbf{k}[\mathbb{G}_{a,1}]$, it is nonzero only when $m = n = 0$, so $\bar{\sigma} = \varepsilon \otimes \varepsilon$.
	
Next we compute $\bar{\tau}$. Since $\eta$ and $\eta^{-1}$ annihilate $\delta_j$ for $j\ge p$, any nonzero contribution to $\overline{\tau}(\delta_n)$ must come from summands in $\Delta^{(2)}(\gamma(\delta_n))=\Delta^{(2)}(\delta_{pn})
=\sum_{a+b+c=pn}\delta_a\otimes\delta_b\otimes\delta_c$ with $a,b,c<p$. Expanding the definition of $\overline{\tau}$ we get $$\overline{\tau}(\delta_n)=\sum_{\substack{a+b+c=pn\\ a,b,c<p}} (-1)^a\sum_{r+s=a}\binom{r+b}{b}\binom{s+c}{c} \delta_{r+b}\otimes\delta_{s+c}\in \mathscr{O}(\mathbb{G}_{a,1})^{\ot 2}.$$ Since $(r+b) + (s+c) = pn$, every term has total degree $pn$. Since the total degree of a simple tensor $\delta_u \otimes \delta_v$ in $\mathscr{O}(\mathbb{G}_{a,1})^{\ot 2}$ is at most $2(p-1)$, $\overline{\tau}(\delta_n) = 0$ for $n \ge 2$.
	
Since $\delta_0$ is the unit in $\mathbf{k}[\mathbb{G}_{a,1}]$, we have $\gamma(\delta_0)=\delta_0$ and $\eta(\delta_0)=\eta^{-1}(\delta_0)=\delta_0$, so $\overline{\tau}(\delta_0)=\delta_0\otimes\delta_0$. All that is left to compute is the case when $n=1$. We have $$\overline{\tau}(\delta_1)=\sum_{\substack{a+b+c=p\\ a,b,c<p}} (-1)^a\sum_{r+s=a}\binom{r+b}{b}\binom{s+c}{c} \delta_{r+b}\otimes\delta_{s+c}.$$ Let us extract the coefficient $\kappa_j$ of $\delta_j\otimes\delta_{p-j}$. Since $a=p-b-c$ and $a < p$, we have $b+c > 0$. Setting $r+b=j$ and $s+c=p-j$ gives
$$\kappa_j =\sum_{\substack{0\le b,c<p\\ 0 < b+c\le p}}
(-1)^{p-b-c}\binom{j}{b}\binom{p-j}{c}.$$ If $1 \le j \le p-1$, then grouping by $m=b+c$ yields $$\kappa_j = \sum_{m=1}^{p}(-1)^{p-m}\sum_{b=0}^m\binom{j}{b}\binom{p-j}{m-b} = \sum_{m=1}^{p}(-1)^{p-m}\binom{p}{m} = 1,$$ where we have used Vandermonde's identity in the second equality. If $j = 0$ or $j = p$, then we cannot have $m = p$ when grouping by $m = b+c$. Thus, in this case $$\kappa_j = \sum_{m=1}^{p-1} (-1)^{p-m}\binom{p}{m} = 0.$$ Hence, $\overline{\tau}(\delta_1)=\sum_{j=1}^{p-1}\delta_j\otimes\delta_{p-j}$, so by (\ref{bz}), we get the result.
\end{proof}

Finally, it follows from Corollary \ref{main0} in a straightforward manner that  
\begin{equation}\label{rz2}
R_{\lambda}:=\sum_{i=0}^{p-1}\frac{\lambda^i}{i!}(t^i\#1)\ot (t^i\#1)
\end{equation}
is an $R$-matrix for $D(\mathbb{G}_{a,1},\mathbb{G}_{a,1},B_{\lambda})= \mathbf{k}[\mathbb{G}_{a,1}^{\vee}]\#^{\tau} \mathbf{k}[\mathbb{G}_{a,2}/\mathbb{G}_{a,1}]=\mathbf{k}[\widetilde{\mathbb{G}_{a,1}}]$ for each $\lambda\in\mathbf{k}$.


\end{document}